\documentclass[12pt,a4paper, leqno]{article}
\usepackage{amssymb, amscd, amsthm, amsmath,amstext}
\usepackage[all]{xy}

\usepackage[colorlinks, linkcolor=blue, anchorcolor=green, citecolor=blue]{hyperref}

\newtheorem{thm}{Theorem}[section]
\newtheorem{coro}[thm]{Corollary}
\newtheorem{lemma}[thm]{Lemma}
\newtheorem{prop}[thm]{Proposition}
\newtheorem{conj}[thm]{Conjecture}

\theoremstyle{remark}
\newtheorem{remark}[thm]{\textbf{Remark}}

\theoremstyle{definition}
\newtheorem{defn}[thm]{Definition}
\newtheorem{example}[thm]{Example}

\numberwithin{equation}{thm}
\newcommand{\newpara}{\noindent\refstepcounter{thm}{\bf(\thethm)\;}}

\newcommand{\C}{\mathbb C}

\newcommand{\N}{\mathbb N}
\newcommand{\Q}{\mathbb{Q}}
\newcommand{\R}{\mathbb R}
\newcommand{\Z}{\mathbb Z}

\newcommand{\al}{\alpha}

\newcommand{\Cor}{\mathrm{Cor}}
\newcommand{\car}{\mathrm{char}}

\newcommand{\ind}{\mathrm{ind}}
\newcommand{\per}{\mathrm{per}}
\newcommand{\sd}{\mathrm{sd}}
\newcommand{\cd}{\mathrm{cd}}
\newcommand{\Br}{\mathrm{Br}}
\newcommand{\Nrd}{\mathrm{Nrd}}
\newcommand{\lra}{\longrightarrow}
\newcommand{\ov}[1]{\overline{#1}}
\newcommand{\simto}{\xrightarrow{\sim}}

\newcommand{\cA}{\mathcal{A}}
\newcommand{\cB}{\mathcal{B}}
\newcommand{\cR}{\mathcal{R}}
\newcommand{\cS}{\mathcal{S}}

\begin{document}
\title{\textbf{On the Rost divisibility of henselian discrete valuation fields of cohomological dimension 3}}
\author{Yong HU and Zhengyao WU}

\maketitle

\begin{abstract}
Let $F$ be a field, $\ell$  a prime and  $D$  a central division $F$-algebra of $\ell$-power degree. By the Rost kernel of $D$ we mean the subgroup of $F^*$ consisting of elements $\lambda$ such that the cohomology class $(D)\cup (\lambda)\in H^3(F,\,\mathbb{Q}_{\ell}/\Z_{\ell}(2))$  vanishes. In 1985, Suslin conjectured that the Rost kernel
is generated by $i$-th powers of reduced norms from  $D^{\otimes i}$ for all $i\ge 1$. Despite of known counterexamples, we prove some new cases of Suslin's conjecture. We assume $F$ is a henselian discrete valuation field with  residue field $k$ of characteristic different from $\ell$. When $D$ has period $\ell$, we show that Suslin's conjecture holds if either   $k$ is a $2$-local field or  the cohomological $\ell$-dimension $\mathrm{cd}_{\ell}(k)$ of $k$ is $\le 2$. When the period is arbitrary, we prove the same result when $k$ itself is a henselian discrete valuation field with $\mathrm{cd}_{\ell}(k)\le 2$. In the case $\ell=\car(k)$ an analog is obtained for tamely ramified algebras. We conjecture that Suslin's conjecture holds for all fields of cohomological dimension 3.
\end{abstract}



\noindent {\bf Key words:} Reduced norms, division algebras over henselian fields, Rost invariant,  biquaternion algebras

\medskip

\noindent {\bf MSC classification:} 11S25, 17A35, 11R52, 16K50


\section{Introduction}

 Let $F$ be  a field and $\ell$ a prime number. For simplicity we first assume $\ell$ is different from the characteristic of $F$.  For an integer $d\ge 1$, we write $H^d(F)=H^d(F,\,\mathbb{Q}_{\ell}/\Z_{\ell}(d-1))$, the inductive limit of the Galois cohomology groups $H^d(F,\,\mu_{\ell^r}^{\otimes (d-1)}),\,r\ge 1$. An element $\al\in H^2(F)$ may be identified with a Brauer class in the $\ell$-primary torsion part of the Brauer group  $\Br(F)$.  Taking the cup product with $\al$ yields a well defined group homomorphism
 \[
R_{\al}\;:\;\; F^*\lra H^3(F)\;;\quad\lambda\longmapsto \al\cup (\lambda)\,.
 \]The kernel of  $R_{\al}$  will be called the \emph{Rost kernel} of $\al$ {and denoted by $\cR(\al)$}. Write $\Nrd(\al):=\Nrd(D^*)$, where $D$ is the central division $F$-algebra in the Brauer class $\al$ and $\Nrd: D^*\to F^*$ denotes the reduced norm map on the nonzero elements of $D$. It is well known that  $\Nrd(\al)$ is contained in the Rost kernel of $\al$.

We remark that when  $\ell=\car(F)$,  the groups $H^d(F)=H^d(F,\,\mathbb{Q}_{\ell}/\Z_{\ell}(d-1)),\,d\ge 1$ and the cup product maps
$H^d(F)\times F^*\to H^{d+1}(F)\,;\;(\theta,\,\lambda)\mapsto \theta\cup(\lambda)$ can still be defined in an appropriate way, and what we said above remains valid (see e.g. \cite[$\S$3.2]{KatoII80}, \cite{Kato82}, or \cite[Appendix\;A]{Merkurjev03inGMScohinv}).

\medskip

In many interesting cases it is known that {the Rost kernel consists precisely of the reduced norms}.  We collect some most important results in this direction.

\begin{example}\label{1p1nnn}
Suppose as above that the period of $\al\in\Br(F)$ is a power of $\ell$. In all the following cases we have {$\cR(\al)=\Nrd(\al)$}.

\begin{enumerate}
\item The (Schur) index of $\al$ is $\ell$.

 This case follows from a well known theorem of Merkurjev and Suslin  (\cite[Theorem\;12.2]{MS82NormResidue}, \cite[Theorems\;24.4]{Suslin85}) in the case $\ell\neq\car(F)$ and its $p$-primary counterpart (\cite[p.94, Thm.\;6]{Gille00}) in the case $\ell=\car(F)$.
  \item The \emph{separable $\ell$-dimension} $\sd_{\ell}(F)$ of $F$ is $\le 2$. (\cite[Thm.\;24.8]{Suslin85} and \cite[p.94, Thm.\;7]{Gille00}.)

  The notion $\sd_{\ell}(F)$ is introduced by Gille in \cite[p.62]{Gille00} (cf. Remark\;\ref{5p4nnn}). If $\ell\neq\car(F)$, it is the same as the cohomological $\ell$-dimension considered in \cite{SerCohGal94}.

  \item $F$ is a global field.

  If $\ell\neq 2$ or if $F$ has no real places (e.g. $F$ is a global function field), this is covered by the previous case. When $F$ is a real number field and $\ell=2$, this follows from a Hasse principle proved by Kneser \cite[Chapter\;5, Theorem 1.a]{Kn69}.
  \item $F$ is a \emph{$2$-local field}, i.e., a complete discrete valuation field whose residue field $k$ is a non-archimedean local field (with finite residue field).

      This case is a consequence of Kato's two-dimensional local class field theory (\cite[p.657, $\S$3.1, Thm.\;3 (2)]{KatoII80}). The case with $\ell\neq\car(k)$ is also observed in \cite[Thm.\;1.3]{ParimalaPreetiSuresh2018}.

  \item The following cases are established in a recent work of Parimala, Preeti and Suresh \cite[Theorems\;1.1 and 1.3]{ParimalaPreetiSuresh2018}:

  \begin{enumerate}
    \item $F$ is a complete discrete valuation field whose  residue field $k$ is a global field of characteristic $\neq \ell$, and if $\ell=2$, suppose $k$ has no real places.

    (If $k$ is a global function field of characteristic $\ell$, the same result is obtained in \cite{Hu19} under the addition hypothesis that $\al$ has period $\ell$ or is tamely ramified. If $k$ is a number field with real places, see Example\;\ref{2p2nnn} (4.b) for a weaker conclusion.)
    \item $F$ is the function field of an algebraic curve over a non-archimedean local field $K$, and $\ell$ is different from the residue characteristic of $K$.
  \end{enumerate}
  \item  $\al$ has index 4, $\car(F)\neq 2$, and  nonsingular quadratic forms of dimension $12$ over all finite extensions of $F$ are all isotropic. (For example, $F$ can be an extension of $\Q_2$ of transcendence degree 1, by \cite[Thm.\;4.7]{PS14invent}.)

  In \cite[Thm.\;4.1]{PreetiSoman15}, this result is stated with the extra assumptions that $F$ has characteristic 0 and cohomological dimension $\le 3$. A careful inspection shows that these assumptions are not used in the proof. (Notice however that the cohomolgical 2-dimension is indeed $\le 3$ under our assumption on the isotropy of 12 dimensional forms.)

  (We will prove a characteristic 2 version in Prop.\;\ref{2p1}.)
\end{enumerate}
\end{example}

In general, the equality $\cR(\al)=\Nrd(\al)$ may not hold. In fact, there exists a field  $F$ of cohomological dimension 3 in characteristic 0 and a biquaternion algebra over $F$ whose Rost kernel contains more elements than the reduced norms (see \cite[Remark\;5.1]{CTPaSu}).  A description of the Rost kernel was conjectured by Suslin as follows:

\medskip

\newpara\label{1p2nnn}\noindent {\bf Suslin's conjecture} (\cite[Conjecture\;24.6]{Suslin85}): If $\al$ has period $\per(\al)=\ell^n$, then
\begin{equation}\label{1p2p1}
\cR(\al)=\prod^n_{i=0}\Nrd(\ell^{i}\al)^{\ell^{i}}=\Nrd(\al)\cdot \Nrd(\ell\al)^{\ell}\cdots \Nrd(\ell^{n-1} \al)^{\ell^{n-1}}\cdot (F^*)^{\ell^n}\;.
\end{equation}
Here for a  multiplicative abelian group $A$ and an integer $m\ge 1$, $A^m$ denotes the subgroup consisting of $m$-th powers in $A$. {We will call the group $\prod^n_{i=0}\Nrd(\ell^{i}\al)^{\ell^{i}}$ the \bf \emph{Suslin group}} { of $\al$  and denote it by $\mathcal{S}(\al)$. Suslin's conjecture amounts to saying that $\cR(\al)=\cS(\al)$. (In \cite{Merkurjev95ProcSympos58} our groups  $\cS(\al)$ and $\cR(\al)$ are denoted by $A(D)$ and $B(D)$ respectively.)}

We may also consider the induced map
\[
\tilde{R}_{\al}\,:\;\; F^*/\Nrd(\al)\lra H^3(F)\;;\quad\lambda\longmapsto \al\cup (\lambda)\,,
\]which is often called the \emph{Rost invariant} map  associated to the semisimple  simply connected algebraic group $\mathbf{SL}_1(D)$ (see e.g. \cite{Merkurjev03inGMScohinv}). In the special case $\per(\al)=\ell$, \eqref{1p2p1} means that the kernel of $\tilde{R}_{\al}$ consists precisely of the $\ell$-divisible elements in the group $F^*/\Nrd(\al)$.

\medskip

We feel that the definition below is now well motivated.

\begin{defn}\label{1p3}
  Let $F$ be a field, $\ell$ a prime number and $n$ a positive integer. We say that $F$ is {\bf \emph{Rost $\ell^n$-divisible}} if Suslin's conjecture holds (i.e. {the Rost kernel equals the Suslin group}) for all $\al\in\Br(F)$ whose period divides $\ell^n$.

  We say that $F$ is {\bf \emph{Rost $\ell^{\infty}$-divisible}} if it is Rost $\ell^n$-divisible for all $n\ge 1$. If $F$ is Rost $\ell^{\infty}$-divisible for all primes $\ell$, we say that $F$ is {\bf \emph{Rost divisible}}. (In that case, {$\cR(\al)=\cS(\al)$ for all Brauer classes $\al$ over $F$}, by Prop.\;\ref{3p4nnn}.)
\end{defn}

Of course, the cases covered by Example\;\ref{1p1nnn} are all examples of Rost divisibility. On the other hand, {in the example where $\cR(\al)\neq \Nrd(\al)$  discussed in \cite[Remark\;5.1]{CTPaSu} $\al$ is the class of a biquaternion algebra and for that $\al$ Suslin's conjecture does hold}. In fact, in characteristic different from 2, {$\cR(\al)$ and $\cS(\al)$ are the same for all Brauer classes $\al$} of degree at most 4. This was remarked by Suslin immediately after the statement of his conjecture (\cite[Conjecture\;24.6]{Suslin85}). His proof, which uses $K$-cohomology groups of Severi--Brauer varieties, can be found in \cite[$\S$1]{Merkurjev95ProcSympos58}. For biquaternion algebras, an alternative proof is given in \cite{KnusLamShapiroTignol95}. In $\S$\ref{sec2} we will extend this last result to characteristic 2 and derive a number of examples of Rost 2-divisibility.

However, not all fields are Rost 2-divisible and Suslin's conjecture can be false for some $\al$ which is a tensor product of three quaternion algebras (see e.g. \cite[p.283]{KnusLamShapiroTignol95}). For an odd prime $\ell$, Merkurjev constructed in \cite[$\S$2]{Merkurjev95ProcSympos58} over a certain field $F$  tensor products of two cyclic algebras of degree $\ell$ that violate Suslin's conjecture. We notice that in these known counterexamples the cohomological dimension $\cd_{\ell}(F)$ must be greater than 3 (see $\S$\ref{sec6} for some more details). We therefore conjecture that Suslin's conjecture is true whenever $\cd_{\ell}(F)\le 3$ if $\car(F)\neq\ell$ (see Conjecture\;\ref{6p2} for a more precise statement).

The goal of this paper is to provide some evidence to our conjecture.

\begin{thm}\label{1p4nnn}
  Let $F$ be a henselian (e.g. complete) discrete valuation field with residue field $k$. Let $\ell$ be a prime number different from $\car(k)$. Suppose that the following properties hold for every finite cyclic extension $L/k$  of degree $1$ or $\ell$:
  \begin{enumerate}
  \item (\emph{Rost $\ell$-divisibility}) $L$ is Rost $\ell$-divisible.
    \item ($H^3$-\emph{corestriction injectivity})  The corestriction map
  \[
  \Cor_{L/k}\,:\;\;H^3(L\,,\,\mathbb{Q}_{\ell}/\Z_{\ell}(2))\lra H^3(k\,,\,\mathbb{Q}_{\ell}/\Z_{\ell}(2))
  \]is injective.
  \end{enumerate}

  Then the field $F$ is Rost $\ell$-divisible.
\end{thm}

The proof of Theorem\;\ref{1p4nnn} will be given in $\S$\ref{sec4} (page\;\pageref{proof1p4}). We will see in Remark\;\ref{6p3} that Condition 2 in this theorem cannot be dropped, although it is not a necessary condition for $F$ to be Rost $\ell$-divisible. {On the other hand, in Theorem\;\ref{4p9v2} we will give a description of the quotient group $\cR(\al)/\cS(\al)$ (the Rost kernel modulo the Suslin group) that is valid without the two assumptions of Theorem\;\ref{1p4nnn}.}

Theorem\;\ref{1p4nnn} also has a version in the case $\ell=\car(k)$. But more definitions and technical assumptions have to be introduced for a proper statement of that version. So we postpone it to $\S$\ref{sec5} (see Theorem\;\ref{5p3nnn}).

\begin{coro}\label{1p5nnn}
Let $F,\,k$ and $\ell$ be as in Theorem$\;\ref{1p4nnn}$. Assume that the residue field $k$ satisfies one of the following conditions:

\begin{enumerate}
  \item $k$ has cohomological $\ell$-dimension $\cd_{\ell}(k)\le 2$;
  \item $k$ is a $2$-local field;
  \item $k=k_0(\!(x)\!)(\!(y)\!)(\!(z)\!)$, where $k_0$ is an algebraically closed field (of characteristic $\neq \ell$).
\end{enumerate}

Then $F$ is Rost $\ell$-divisible.
\end{coro}
\begin{proof}
It is sufficient to show that $k$ has the two properties stated in Theorem\;\ref{1p5nnn}. The case with $\cd_{\ell}(k)\le 2$ is clear (see Example\;\ref{1p1nnn} (2) for the Rost divisibility). When $k$ is a 2-local field, the Rost divisibility condition is satisfied according to Example\;\ref{1p1nnn} (4) and the corestriction injectivity property holds by \cite[p.660, $\S$3.2, Prop.\;1]{KatoII80}. In the third case, the injectivity of corestriction can be shown in the same way as for 2-local fields. The Rost $\ell$-divisibility of $k$ follows from Case 1.
\end{proof}

In the above corollary the field $k$ can also be a number field or the field  $\R(\!(x)\!)$. In that situation  only the case $\ell=2$ needs to be treated. For a number field the $H^3$-corestriction injectivity can be deduced from a result of Tate on the cohomology of global fields (cf. \cite[Cor.\;8.3.12 (iii)]{NeukirchCohNumField08}). However, we feel that a simpler method in that case is to utilize the theory of quadratic forms (see Example\;\ref{2p2nnn} (4.b)).

We think that whether the field $F$ in Corollary\;\ref{1p5nnn} is Rost $\ell^{\infty}$-divisible is an interesting open problem. In general, we wonder if there is an approach to pass from the Rost $\ell$-divisibility to the $\ell^{\infty}$-divisibility.

A special case is treated in the following  theorem, which we will prove in the second half of $\S$\ref{sec4} (see \eqref{4p5nnn}--\eqref{4p8nnn}).

\begin{thm}\label{1p6nnn}
Let $\ell$ be a prime number and $n\in\N^*$. Let $k$ be a henselian discrete valuation field of residue characteristic different from $\ell$. Let $F$  be a henselian discrete valuation field with residue field $k$.

If $\cd_{\ell}(k)\le 2$ and $\mu_{\ell^n}\subseteq k$ (i.e., $k$ contains a primitive $\ell^n$-th root of unity), then $F$ is Rost $\ell^n$-divisible.
\end{thm}

For example, the field  $\C(x)(\!(y)\!)(\!(z)\!)$ is Rost divisible according to the above theorem.

\section{Rost kernel of biquaternion algebras}\label{sec2}

As a warmup, we extend Suslin's conjecture for biquaternion algebras to the characteristic 2 case.

\medskip

Recall that by a theorem of Albert, a biquaternion algebra over any field $F$ is the same as a central simple $F$-algebra of period 2 and degree 4 (see e.g. \cite[Thm.\;16.1]{KMRT}).

\begin{prop}\label{2p1}
Let $F$ be a field and let $ \alpha \in\Br(F)$ have period $2$ and index $4$.
Suppose $ \lambda\in F^{*} $ satisfies $ \alpha\cup(\lambda)=0 $ in $ H^{3}(F)$.

Then $\lambda\in (F^{*})^{2}\cdot\Nrd(\alpha)$. If moreover the $u$-invariant of $F$ is less than $12$ (meaning that every nonsingular quadratic form of dimension $12$ over $F$ is isotropic), then $\lambda\in \Nrd(\al)$.
\end{prop}

As was mentioned in the introduction, the proof of the above result in characteristic different from 2 already exists in the literature. The first assertion can be found in   \cite{KnusLamShapiroTignol95} or \cite{Merkurjev95ProcSympos58}, and the second one follows from \cite[Thm.\;4.1]{PreetiSoman15} (see also our remarks in Example\;\ref{1p1nnn} (6)).

\medskip

To complete the proof of Prop.\;\ref{2p1}, we need to use basic facts about quadratic forms and cohomology theories in characteristic 2 which we now recall. The reader may consult \cite[Chapters\;1 and 2]{EKM08} and \cite{Kato82} for more details.

Let $F$ be a field of characteristic 2. For any $a,\,b\in F$, we denote by $[a,\,b]$ the binary quadratic form $(x,\,y)\mapsto ax^2+xy+by^2$. A \emph{quadratic $1$-fold Pfister form} over $F$ is a quadratic form of the type $[1,\,b]$ for some $b\in F$. For $n\ge 2$, a \emph{quadratic $n$-fold Pfister form} is a tensor product of a bilinear $(n-1)$-fold Pfister form with a quadratic 1-fold Pfister form, i.e., a form of the shape
\[
[\![b\,;\,a_1,\dotsc, a_{n-1}\rangle\!\rangle :=[1\,,\,b]\otimes\langle\!\langle a_1,\dotsc, a_{n-1}\rangle\!\rangle
\]
The group of Witt equivalence classes of nonsingular (even dimensional) quadratic forms over $F$ will be denoted by $I_q(F)$ or $I^1_q(F)$, and for each $n\ge 1$, $I_q^n(F)$ denotes the subgroup of $I_q(F)$ generated by scalar multiples of quadratic $n$-fold Pfister forms.

For a natural number $r\in\N$, we have the Kato--Milne cohomology group $H_2^{r+1}(F)=H^{r+1}(F\,,\,\Z/2\Z(r))$, which can be described using absolute differentials. (A brief review about these groups will be given in \eqref{5p1nnn}.) For each $n\ge 1$ there is a well defined group homomorphism $e_n: I^n_q(F)\lra H_2^n(F)$ with the property that (cf. \cite[p.237, Prop.\;3]{Kato82})
\begin{equation}\label{2p1p1}
  e_n(\varphi)=e_n(\psi) \iff \varphi\cong \psi
\end{equation}
for all quadratic $n$-fold Pfister forms $\varphi$ and $\psi$ over $F$.

\

\begin{proof}[Proof of Proposition$\;\ref{2p1}$ in characteristic $2$]
Let $[b,\,a)\otimes [d,\,c)$ be a biquaternion algebra representing the Brauer class $\al$, where $[b,\,a)$ denotes the quaternion $F$-algebra generated by two elements $x,\,y$ subject to the relations
\[
x^2-x=b\,,\;y^2=a\quad\text{and}\quad yx=(x+1)y\,.
\]Then the quadratic form $\phi:=[1\,,\,b+d]\bot a.[1,\,b]\bot c.[1,\,d]$ is an Albert form of $\al$ (cf. \cite[(16.4)]{KMRT}). Inside the group $I_q(F)$, we have (cf. \cite[Example\;7.23]{EKM08})
   \[
   \phi=[1,\,b]+a.[1,\,b]+[1,\,d]+c.[1,\,d]=[\![b,\,a\rangle\!\rangle -[\![d,\,c\rangle\!\rangle\;\in I^2_q(F)\,
  \]and for any $\lambda \in F^*$,
  \[
   \phi-\lambda. \phi= \phi\otimes \langle 1,\,-\lambda\rangle=[\![b,\,a\,,\,\lambda\rangle\!\rangle -[\![d,\,c\,,\lambda\rangle\!\rangle\;\in I^3_q(F)\,.
  \]The cohomological invariant $e_3: I^3_q(F)\to H^3(F,\,\Z/2\Z(2))$ sends $[\![b,\,a\,,\,\lambda\rangle\!\rangle -[\![d,\,c\,,\lambda\rangle\!\rangle$ to the cohomology class
  $\al\cup (\lambda)\in H^3(F,\,\Z/2\Z(2))$. When $\al\cup(\lambda)=0$, we can deduce from \eqref{2p1p1} and \cite[(16.6)]{KMRT} that
  \[
   \lambda\;\in\; G(\phi):=\{\rho\in F^*\,|\,\rho.\phi\cong \phi\}=F^{*2}\cdot\Nrd(\al)\;.
  \]This proves the first assertion.

  Now assume further that every nonsingular quadratic form over dimension 12 over $F$ is isotropic. Then for the Albert form $\phi$ and any $\rho\in F^*$, the form $\langle 1,\,-\rho\rangle\otimes\phi$ is isotropic and hence $\rho$ is a spinor norm for $\phi$. By \cite[Prop.\;16.6]{KMRT}, we have $\rho^2\in\Nrd(\al)$. This shows $F^{*2}\subseteq\Nrd(\al)$. So the second assertion follows.
\end{proof}

\begin{example}\label{2p2nnn}
  Prop.\;\ref{2p1} implies that a field $F$ is Rost 2-divisible if  every Brauer class in $\Br(F)[2]$ has index at most 4.  The following fields $F$ possess this property:

  \begin{enumerate}
    \item Any field $F$ of characteristic $2$ with $[F:F^2]\le 4$.

    This is a well known theorem of Albert (cf. \cite[Lemma\;9.1.7]{GilleSzamuely17}).
    \item $F=k_0(\!(x)\!)(\!(y)\!)(\!(z)\!)$ is an iterated Laurent series field in three variables over quasi-finite field $k_0$ of characteristic $2$.

Here by a \emph{quasi-finite} field we mean a perfect field whose absolute Galois is isomorphic to that of a finite field. That every Brauer class in $\Br(F)[2]$ is a biquaternion algebra is proved in \cite[Thm.\;3.3]{AravireJacob95}.
    \item A field extension $F$ of  transcendence degree 2 over any finite field (\cite{Lieblich15AnnMath}).
    \item A complete discrete valuation field $F$ whose residue field $k$ satisfies one of the following hypotheses:

    \begin{enumerate}
      \item $k$ has characteristic $2$ and $[k:k^2]\le 2$ (cf. \cite[Thm.\;2.7]{PS14invent}).
      \item $k$ has characteristic different from $2$ and every 2-torsion Brauer class over $k$ is (the class of) a quaternion algebra.

      This case follows from Witt's decomposition of the 2-torsion Brauer subgroup $\Br(F)[2]$ (see \eqref{4p1p3}).

      Concrete examples of $k$ include:
      \begin{itemize}
        \item[(i)]  a global field or a local field;

      \item[(ii)] any $C_2$-field;

      \item[(iii)] the field $\R(\!(x)\!)$;

      \item[(iv)] a one-variable function field over $\R$;

      \item[(v)] an iterated Laurent series field $k_0(\!(x)\!)(\!(y)\!)(\!(z)\!)$, where $k_0$ is  a quadratically closed field.
      \end{itemize}

      Case (i) is well known. (Notice however that the result here is covered by Example\;\ref{1p1nnn} (4) and (5.a) unless $k$ is a number field with real places.)

      For $C_2$-fields the assertion follows from an easy consideration of the isotropy of Albert forms of biquaterion algebras. For the field $k=\R(\!(x)\!)$, the group $\Br(k)[2]$ can be computed explicitly. The other cases were discussed in \cite[p.126 and Prop.\;3.12]{ElmanLam73Invent}.
    \end{enumerate}

    \item A one-variable function field $F$ over a complete discrete valuation field whose residue field $k$ is a $C_1$-field of characteristic $\neq 2$.

    Indeed, Artin proved in \cite[Thm.\;6.2]{Artin82BrauerSeveri} that period coincides with index for 2-primary torsion Brauer classes over any $C_2$-field. One can apply \cite[Cor.\;5.6]{HHK} or \cite[Thm.\;6.3]{Lieb07} to see that 2-torsion Brauer classes over $F$ have index at most $4$. The case with $k$ finite was first proved by Saltman \cite{Salt97}.
    \item A one-variable function field $F$ over a complete discrete valuation field whose residue field $k$ is a perfect field of characteristic $2$.

    This case is established in \cite[Thm.\;3.6]{PS14invent}.
    \item The field of fractions $F$ of a two-dimensional henselian excellent local domain of finite residue field of characteristic $\neq 2$ (\cite[Thm.\;3.4]{Hu11}).
  \end{enumerate}

Except possibly for Case 2, the cases (iii)--(v) in (4.b) and the case of real number fields in (4.b) (i), the fields $F$ in the above list have $u$-invariant $\le 8$, and hence the second conclusion of Prop.\;\ref{2p1} applies to these fields.

Indeed, Case 1 follows from \cite[p.338, Cor.\;1]{MammoneMoresiWadsworth91}. In Case 3, $F$ is a $C_3$-field. In Cases (4.a) and (4.b) (ii), one can combine \cite[p.338, Cor.\;1]{MammoneMoresiWadsworth91} with Springer's theorem (\cite[XI.6.2 (7)]{Lam} if $\car(k)\neq 2$; \cite[Thm.\;1.1]{Baeza82} and \cite[p.343, Thm.\;2]{MammoneMoresiWadsworth91} if $\car(k)=2$). Cases 5, 6 and 7 can be found in \cite[Cor.\;4.13]{HHK}, \cite[Thm.\;4.7]{PS14invent} and \cite[Thm.\;1.2]{Hu11} respectively.
\end{example}

Note that in the above list we have restricted to examples that do not depend on Theorem\;\ref{1p4nnn} or Corollary\;\ref{1p5nnn}.

\section{The Suslin group}\label{sec3}

In this section, we prove some general properties of the Suslin group of an arbitrary Brauer class and show that the general case can be reduced to the case of prime power degree classes.

\medskip

\newpara\label{3p1nnn} Let $F$ be a field and $\al\in\Br(F)$. By a well known result, called the \emph{norm principle for reduced norms} (cf. \cite[Prop.\;2.6.8]{GilleSzamuely17}),
\[
  \Nrd(\al)=\bigcup_{L/F}N_{L/F}(L^*)=\bigcup_{M/F}N_{M/F}(M^*)\,,
\]
where $L/F$ runs over finite separable extensions splitting $\al$ that can be $F$-embedded into the central division algebra $D_{\al}$ in the Brauer class $\al$, and $M/F$ runs over finite extensions that split  $\al$. As an easy consequence, for all $t\in\Z$ we have
\begin{equation}\label{3p1p1}
     \Nrd(\al)\subseteq\Nrd(t\al) \text{ and equality holds if } \gcd(t,\,\per(\al))=1\,.
\end{equation}
Moreover, for all $\al_1,\,\al_2\in \Br(F)$,
\begin{equation}\label{3p1p2}
  \gcd(\per(\al_1)\,,\,\per(\al_2))=1 \Longrightarrow \Nrd(\al_1+\al_2)=\Nrd(\al_1)\cap\Nrd(\al_2)\,.
\end{equation}
Indeed, a field extension $M/F$ splits $\al_1+\al_2$ if and only if it splits both $\al_1$ and $\al_2$, since the periods of $\al_1$ and $\al_2$ are coprime. So the norm principle shows that
\[
\Nrd(\al_1+\al_2)\subseteq\Nrd(\al_1)\cap \Nrd(\al_2)\,.
\]Conversely, suppose $\lambda\in \Nrd(\al_1)\cap\Nrd(\al_2)$. Letting $D_i$ be the central division $F$-algebra in the Brauer class $\al_i$ for $i=1,\,2$, by the norm principle we can find a subfield $L_i\subseteq D_i$ such that $\lambda\in N_{L_i/F}(L_i^*)$. Let $L=L_1\cdot L_2$ be the composite field (inside a fixed algebraic closure of $F$) and put $d_i=[L_i:F]$. Then we have $\gcd(d_1,\,d_2)=1$ since $d_i$ divides a power of $\per(\al_i)$. In particular, $[L:F]=d_1d_2$. Thus,
\[
\lambda^{d_2}\in N_{L_1/F}(L_1^{*d_2})\subseteq N_{L_1/F}\left(N_{L/L_1}(L^*)\right)\;\subseteq N_{L/F}(L^*)\,,
\]and similarly, $\lambda^{d_1}\in N_{L/F}(L^*)$. Since $\gcd(d_1,\,d_2)=1$, this implies $\lambda\in N_{L/F}(L^*)$. As $L=L_1L_2$ is clearly a splitting field of $\al_1+\al_2$, we get $\lambda\in \Nrd(\al_1+\al_2)$. This proves \eqref{3p1p2}.

\

\newpara\label{3p2nnn} {In \eqref{1p2nnn} we defined the Suslin group of a Brauer class of prime power period. Now we generalize this definition to arbitrary Brauer classes. For any $\al\in\Br(F)$, we define the  {\bf \emph{Suslin group}} $\cS(\al)$, which in \cite{Merkurjev95ProcSympos58} was denoted by $A(D)$, as follows:}
\begin{equation}\label{3p2p1}
  \begin{split}
    \cS(\al)=\prod^{\infty}_{i=1}\Nrd(i\al)^i&:=\text{ subgroup of $F^*$ generated by }\;\bigcup_{i\ge 1}\Nrd(i\al)^i\\
    &=\prod^{\per(\al)}_{i=1}\Nrd(i\al)^i\;.
  \end{split}
\end{equation}
In fact, writing $e=\per(\al)$ we have
\[
\begin{split}
  \cS(\al)&=\prod^e_{i=1}\Nrd(i\al)^i=\prod_{d\,|\,e}\prod_{\substack{1\le i\le e \\ \gcd(i,\,e)=d}}\Nrd(i\al)^i\\
  &=\prod_{d\,|\,e}\prod_{\substack{1\le j\le e/d \\ \gcd(j,\,e/d)=1}}\Nrd(jd\al)^{jd}\\
  &=\prod_{d\,|\,e}\prod_{\substack{1\le j\le e/d \\ \gcd(j,\,e/d)=1}}\Nrd(d\al)^{jd}\quad(\text{using }\eqref{3p1p1})\\
  &=\prod_{d\,|\,e}\Nrd(d\al)^d\,.
\end{split}
\]
Also, from the definition and \eqref{3p1p1} we see easily that for all $t\in\N^*$,
\begin{equation}\label{5p2p3}
   \cS(t\al)^t\subseteq\cS(\al)\subseteq \cS(t\al)\,,   \text{ and }\; \cS(\al)=\cS(t\al)\text{ if }\;\gcd(t,\,\per(\al))=1\,.
\end{equation}

We can prove the following analog of \eqref{3p1p2}:

\begin{lemma}\label{3p3nnn}
  For all $\al_1,\,\al_2\in\Br(F)$, if $\gcd(\per(\al_1),\,\per(\al_2))=1$, then $\cS(\al_1+\al_2)=\cS(\al_1)\cap \cS(\al_2)$.
\end{lemma}
\begin{proof}
  For every $i\ge 1$, we have $\gcd(\per(i\al_1),\,\per(i\al_2))=1$. So by \eqref{3p1p2}, $\Nrd(i\al_1+i\al_2)=\Nrd(i\al_1)\cap\Nrd(i\al_2)$ and hence
  \[
  \Nrd(i\al_1+i\al_2)^i\subseteq \Nrd(i\al_1)^i\cap\Nrd(i\al_2)^i\;.
  \]Together with \eqref{3p2p1} this proves $\cS(\al_1+\al_2)\subseteq\cS(\al_1)\cap \cS(\al_2)$.

  Conversely, suppose $\lambda\in \cS(\al_1)\cap\cS(\al_2)$. Write $e_i=\per(\al_i)$ and $\al=\al_1+\al_2$. Using \eqref{5p2p3} we find
  \[
  \lambda^{e_2}\in \cS(\al_1)^{e_2}=\cS(e_2\al_1)^{e_2}=\cS(e_2\al)^{e_2}\subseteq\cS(\al)
  \]and similarly, $\lambda^{e_1}\in \cS(\al)$. Since $\gcd(e_1,\,e_2)=1$, a standard argument yields $\lambda\in \cS(\al)=\cS(\al_1+\al_2)$. This proves the lemma.
\end{proof}

The result below is immediate from Lemma\;\ref{3p3nnn}.

\begin{prop}\label{3p4nnn}
  Let $F$ be a field and $N$ a positive integer. If $F$ is Rost $\ell^n$-divisible for every prime power $\ell^n$ that divides $N$, then for all $\al\in \Br(F)[N]$ we have
{$\cR(\al)=\cS(\al)$, i.e., the Rost kernel and the Suslin group}  of $\al$ coincide.

  Consequently, if $F$ is Rost divisible (i.e., Rost $\ell^{\infty}$-divisible for every prime $\ell$), then the Rost kernel and the {Suslin group} are the same for all $\al\in\Br(F)$.
\end{prop}

This proposition shows that our reformulation of Suslin's conjecture (cf. \eqref{1p2nnn}) is equivalent to the original one given in \cite[(24.6)]{Suslin85}.

\begin{remark}\label{3p5nnn}
{The Suslin group $\cS(\al)$ and the Rost kernel $\cR(\al)$} may be interpreted from a $K$-cohomological point of view.  See e.g. \cite[(1.6) and (1.10)]{Merkurjev95ProcSympos58}.
\end{remark}

\section{Proofs of main results}\label{sec4}

We prove our main theorems\;\ref{1p4nnn} and \ref{1p6nnn} in this section.

\medskip

\newpara\label{4p1nnn} Throughout this section we use the following notations:

\begin{enumerate}
  \item $F$: a henselian  discrete valuation field
  \item $v=v_F$: the normalized discrete valuation on $F$
  \item $k$: the residue field of $F$
  \item $\pi\in F$: a fixed uniformizer of $F$
  \item $\ell$: a prime number different from the characteristic of $k$
  \item $\al\in \Br(F)[\ell^{\infty}]$: a Brauer class over $F$ of $\ell$-power period
  \item $H^d(\cdot):=H^d(\cdot\,,\,\mathbb{Q}_{\ell}/\Z_{\ell}(d-1))$ for all $d\in\N^*$, the cohomology being the Galois cohomology.
  \item {For any finite extension $L$ of $F$, let $\mathcal{O}_L$ be the valuation ring of $L$ and $U_L$ be the group of units in $\mathcal{O}_L$}.
\end{enumerate}

An element $\chi_0 \in H^1(k)$ can be determined by a pair $(E_0/k,\,\ov{\sigma})$, where $E_0/k$ is the cyclic extension and $\ov{\sigma}$ is a  generator of the cyclic Galois group $\mathrm{Gal}(E_0/k)$. The correspondence between $\chi_0$ and $(E_0/k,\,\ov{\sigma})$ is established by requiring  that the continuous homomorphism $\chi_0:\mathrm{Gal}(k_s/k)\to \mathbb{Q}_{\ell}/\Z_{\ell}$ has kernel $\mathrm{Gal}(k_s/E)$, $k_s$ denoting a fixed separable closure of $k$, and that $\ov{\sigma}\in \mathrm{Gal}(E_0/k)$ is the generator which is mapped to the canonical generator of the cyclic group $\mathrm{Im}(\chi_0)$. When the role played by $\ov{\sigma}$ is not explicit in our arguments, we will simply write $\chi_0=(E_0/k)$.

By the \emph{canonical lifting} $\chi\in H^1(F)$ of $\chi_0$ we shall mean the image of $\chi_0$ under the inflation map $H^1(k)\to H^1(F)$. Explicitly, $\chi$ is defined by the pair $(E/F,\,\sigma)$ where $E/F$ is the unramified extension with residue field extension $E_0/k$ and $\sigma\in\mathrm{Gal}(E/F)$ is the generator corresponding to $\ov{\sigma}$ via the natural isomorphism $\mathrm{Gal}(E/F)\simto \mathrm{Gal}(E_0/k)$. Just as for $\chi_0$, we will write $\chi=(E/F)$ for short.

For any element $b\in F^*$, we write
\[
(E/F,\,b)=\chi\cup(b)
\]for the Brauer class given by the cup of $\chi\in H^1(F)$ and $(b)\in H^1(F,\,\mathbb{Q}_{\ell}/\Z_{\ell}(1))$.

For each $d\ge 1$, there is a well known homomorphism, called the \emph{residue map},
\[
\partial\;:\;\;H^{d+1}(F)\lra H^d(k)
\]whose definition may differ by a sign in different references (cf. \cite[$\S$1.2]{KatoII80}, \cite[$\S$7]{Serre03inGMScohinv}). Here we are mostly interested in the residue maps defined on $H^2$ and $H^3$. We make our choice of sign in a way that the following formulas hold: For all $\lambda,\,\mu\in F^*$ and $\chi_0\in H^1(k)$ with canonical lift $\chi\in H^1(F)$, we have
\begin{equation}\label{4p1p1}
  \partial(\chi\cup(\lambda))=v(\lambda).\chi_0\;\in\;H^1(k)
\end{equation}
and
\begin{equation}\label{4p1p2}
\partial(\chi\cup(\lambda)\cup(\mu))=\chi_0\cup (-1)^{v(\lambda)v(\mu)}\ov{\lambda^{v(\mu)}\mu^{-v(\lambda)}}\;\in\;H^2(k)\,.
\end{equation}

Moreover, we have an exact sequence
\begin{equation}\label{4p1p3}
0\lra \Br(k)[\ell^{\infty}]=H^2(k) \overset{\iota}{\lra} \Br(F)[\ell^{\infty}]=H^2(F)\overset{\partial}{\lra} H^1(k)\lra 0\,,
\end{equation}for which the choice of the uniformizer $\pi$ determines a splitting
\[
H^1(k)\lra H^2(F)\,;\,\quad \chi_0\longmapsto \chi\cup (\pi)\;.
\]The map $\iota$ can be viewed as the inflation map $H^2(k)\to H^2(F)$, or the composition of the natural map $\Br(\mathcal{O}_F)[\ell^{\infty}]\to \Br(F)[\ell^{\infty}]$ with the inverse of the natural isomorphism $\Br(\mathcal{O}_F)[\ell^{\infty}]\cong \Br(k)[\ell^{\infty}]$, where $\mathcal{O}_F$ is the valuation ring of $F$. A Brauer class in the kernel of the residue map $\partial: H^2(F)\to H^1(k)$ is
 called \emph{unramified}.

For $\al\in \Br(F)[\ell^{\infty}]$, we may write
\begin{equation}\label{4p1p4}
  \al=\al'+(E/F\,,\,\pi)\;\in\;\Br(F)\quad\text{ with }\;\al'\in \Br(F) \text{ unramified}
\end{equation}where $(E/F)=\chi$ is the canonical lifting of $\chi_0:=\partial(\al)\in H^1(k)$,

For any $\lambda\in F^*$, if we write
\begin{equation}\label{4p1p5}
  \lambda=\theta.(-\pi)^r\quad \text{with }\; r\in\Z \text{ and } \theta \text{ a unit for the valuation of } F\,
\end{equation}then  computation of residues shows that (cf. \cite[Lemma\;4.7]{ParimalaPreetiSuresh2018})
\begin{equation}\label{4p1p6}
\begin{split}
  &\al\cup(\lambda)\;\in\;\ker\left(\partial: H^3(F)\to H^2(k)\right)\\
\iff & r\al'=(E/F\,,\,\theta)\\
\iff &r\al=(E/F\,,\,(-1)^r\lambda)\;.
\end{split}
\end{equation}

\begin{lemma}\label{4p2nnn}
  With notation as above, if $r=v_F(\lambda)$ is coprime to $\ell$, then
  \[
  \partial(\al\cup(\lambda))=0 \Longrightarrow \lambda\in \Nrd(\al)\,.
  \]
\end{lemma}
\begin{proof}
  Let $\ell^n=\mathrm{ind}(\al)$ be the index of $\al$. There exist integers $s,\,c\in \Z$ such that $rs+\ell^nc=1$. Replacing $\lambda$ with $\lambda^s\pi^{\ell^nc}$, we may assume $r=v_F(\lambda)=1$. Then by \eqref{4p1p6},
  \[
       \partial(\al\cup(\lambda))=0  \Longrightarrow \al=(E/F\,,\,-\lambda)\Longrightarrow \al \text{ is split by } L:=F\left(\sqrt[\ell^n]{-\lambda}\right)\,.
  \]Hence,
  \[
  (-1)^{\ell^n}\lambda=N_{L/F}\left(\sqrt[\ell^n]{-\lambda}\right)\;\in\; N_{L/F}(L^*)\subseteq\Nrd(\al)\,,
  \]and the lemma is thus proved.
\end{proof}

\begin{coro}\label{4p3v2}
  With notation as in $\eqref{4p1nnn}$, if $\per(\al)=\ell$, then
  \[
  \cR(\al)=\cS(\al)\cdot (\cR(\al)\cap U_F)\quad\text{and}\quad \frac{\cR(\al)\cap U_F}{\cS(\al)\cap U_F}\cong \frac{\cR(\al)}{\cS(\al)}\,.
  \]
\end{coro}
\begin{proof}
  Clearly, it is sufficient to prove $\cR(\al)\subseteq\cS(\al)\cdot(\cR(\al)\cap U_F)$.

  Suppose $\lambda\in \cR(\al)$ and write $r=v_F(\lambda)$. If $r$ is coprime to $\ell$, then by Lemma\;\ref{4p2nnn} we have $\lambda\in \Nrd(\al)\subseteq \cS(\al)$. Otherwise $r$ is a multiple of $\ell=\per(\al)$.  In particular, this implies $\al\cup ((-\pi)^r)=0$. Using the notation of \eqref{4p1p5}, we get $\al\cup (\theta)=0$, i.e., $\theta\in \cR(\al)\cap U_F$. Hence
  \[
  \lambda=(-\pi)^r.\theta\;\in\;(F^*)^{\ell}\cdot(\cR(\al)\cap U_F)\subseteq \cS(\al)\cdot (\cR(\al)\cap U_F)
  \]as desired.
\end{proof}

\begin{lemma}\label{4p3nnn}
  With notation as in $\eqref{4p1p4}$, let $\ov{\al}'\in\Br(k)$ be the canonical image of the unramified class $\al'$.

  \begin{enumerate}

    \item For the unramified Brauer class $\al'$, we have
    \begin{equation}\label{4p4p1v2}
      \cR(\al')=(\cR(\al')\cap U_F)\cdot (F^*)^{\per(\al')}\;,\quad \cS(\al')=(\cS(\al')\cap U_F)\cdot (F^*)^{\per(\al')}
    \end{equation}
    and
   \begin{equation}\label{4p4p2v2}
    \cR(\al')\cap U_F=\{a\in U_F\,|\,\ov{a}\in \cR(\ov{\al}')\}\;,\;\cS(\al')\cap U_F=\{a\in U_F\,|\,\ov{a}\in \cS(\ov{\al}')\}\,.
    \end{equation}
    \item  If $\ell^n=\per(\al)$ and  $k$ is Rost $\ell^n$-divisible, then $\cR(\al')=\cS(\al')$.
    \item  Letting $E_0$ denote the residue field of $E$, we have
    \begin{equation}\label{4p4p3v2}
    \cR(\al)\cap U_F=\cR(\al')\cap N_{E/F}(U_E)=\{a\in U_F\,|\,\bar a\in \cR(\ov{\al}')\cap N_{E_0/k}(E_0^*)\}\,.
    \end{equation}
  \end{enumerate}
\end{lemma}
\begin{proof}
(1) Let $\ell^n=\per(\al')$ and $\lambda\in F^*$. If $\lambda\in \cR(\al')$, by computing the residue of the cohomology class $\al'\cup(\lambda)$ we see that $r=v_F(\lambda)$ is a multiple of $\ell^n=\per(\al')$.  Thus, using the notation of \eqref{4p1p5}, we get
  \[
  \al'\cup(\theta)=\al'\cup(\theta.(-\pi)^r)=\al'\cup(\lambda)=0\,.
  \]This means $\theta\in \cR(\al')\cup U_F$. The first equality in \eqref{4p4p1v2} is thus proved. Since $(F^*)^{\ell^n}$ is contained in $\cS(\al')$ by the definition of the Suslin group, this argument also proves the other equality in \eqref{4p4p1v2}.

For any $a\in U_F$, the cohomology class $\al'\cup(a)$ is unramified. So its specialization $\ov{\al}'\cup (\ov{a})$ vanishes in $H^3(k)$ if and only if $\al'\cup (a)$ vanishes in $H^3(F)$. This gives the first equality in  \eqref{4p4p2v2}.

To show the second equality in \eqref{4p4p2v2}, let us suppose $\theta\in U_F$.

If $\ov{\theta}\in \cS(\ov{\al}')$, we have
  \[
  \ov{\theta}\;\in\; \prod^n_{i=0}\Nrd(\ell^{n-i}\ov{\al}')^{\ell^{n-i}}\;.
  \]Thus, there exist finite (separable) splitting fields $l_i/k$ of $\ell^{n-i}\ov{\al}'$ for $i=0,\,1,\dotsc, n$ and elements $\bar u_i\in l_i^*$ such that
  \[
  \bar \theta=\prod^n_{i=0}N_{l_i/k}(\bar u_i)^{\ell^{n-i}}\;.
  \]

   Let $L_i/F$ be the unramified extension with residue field $l_i/k$. Then $\ell^{n-i}\al'_{L_i}=0\in \Br(L_i)$ since $\ell^{n-i}\ov{\al}'_{l_i}=0$. This implies that $N_{L_i/F}(L_i^*)\subseteq \Nrd(\ell^{n-i}\al')$. Lift $\bar u_i\in l_i^*$ to a unit $u\in L_i^*$. Then
  \[
  \theta^{-1}.\prod^n_{i=0}N_{L_i/F}(u_i)^{\ell^{n-i}}\;\subseteq\;U^1_F:=\{a\in F^*\,|\,v_F(1-a)\ge 1\}\,.
  \]Since $U_F^1\subseteq (F^*)^{\ell^n}=\Nrd(\ell^n\al')^{\ell^n}$ by Hensel's lemma, it follows that
  \[
  \theta\in \prod^n_{i=0}\Nrd(\ell^{n-i}\al')^{\ell^{n-i}}=\cS(\al')\;.
  \]

 Now suppose $\theta\in \cS(\al')\cap U_F$. To finish the proof, we must  show $\ov{\theta}\in\cS(\ov{\al}')$.

 Notice that for any unramified Brauer class $\beta\in\Br(F)[\ell^{\infty}]$, with canonical image $\ov{\beta}\in \Br(k)$, if $M/F$ is a splitting field of $\beta$ of $\ell$-power degree, then the maximal unramified subextension $M_0$ of $M$ also splits $\beta$ (cf. \cite[(8.4)]{Serre03inGMScohinv}). From this remark and the assumption $\theta\in \cS(\al')$ we can deduce that
 \[
 \theta=\prod_{i=0}^nN_{L_i/F}(u_i)\pi^{t_i}
 \]where $L_i/F$ is an unramified extension of $\ell$-power degree splitting  $\ell^{n-i}\al'$, $u_i$ is a unit in $L_i$ and $t_i\in\Z$. Since $\theta\in U_F$, we have
 \[
 \sum_{i=0}^nt_i=0\quad\text{and hence }\;  \theta=\prod_{i=0}^nN_{L_i/F}(u_i)\,.
 \]For each $i$, the residue field $l_i$ of $L_i$ splits $\ell^{n-i}\ov{\al}'$. So it follows that
 \[
 \ov{\theta}\;\in\;\prod_{i=0}^nN_{l_i/k}(\ov{u}_i)=\cS(\ov{\al}')\,.
 \]

  (2) The Rost $\ell^n$-divisibility of $k$ implies that $\cR(\ov{\al}')=\cS(\ov{\al}')$. So the result follows immediately from (1).

  (3) Note that $[E:F]$ is a power of $\ell$ and hence invertible in the residue field $k$. Thus, by Hensel's lemma $U_F^1=\{a\in U_F\,|\,\ov{a}=\bar 1\}$ is contained in $N_{E/F}(U_E)$. This implies
  \[
  U_F\cap N_{E/F}(E^*)=N_{E/F}(U_E)=\{a\in U_F\,|\,\ov{a}\in N_{E_0/k}(E_0^*)\}\,.
  \]We have seen $\cR(\al')\cap U_F=\{a\in U_F\,|\,\ov{a}\in \cR(\ov{\al}')\}$ in (1). So the second equality in (3) follows.

  If $\theta\in N_{E/F}(E^*)$, we have $(E/F,\,\pi)\cup (\theta)=0$ since $(E/F)\cup (\theta)=0$ in $H^2(F)$. Hence, when $\theta\in \cR(\al')\cap N_{E/F}(U_E)$, we have
  \[
  \al\cup (\theta)=\al'\cup(\theta)+(E/F,\,\pi)\cup (\theta)=0\,,
  \]whence $\theta\in \cR(\al)$.

  Conversely, if $\theta\in \cR(\al)\cap U_F$, then we have $(E/F,\,\theta)=0\in H^2(F)$ by \eqref{4p1p6}. On the one hand, this implies $(E/F,\,\pi)\cup (\theta)=0$, so that from $\theta\in \cR(\al)$ we obtain $\theta\in \cR(\al')$. On the other hand, it follows that $\theta\in U_F\cap N_{E/F}(E^*)=N_{E/F}(U_E)$. This completes the proof.
\end{proof}

\begin{lemma}\label{4p4nnn}
  With notation as in $\eqref{4p1nnn}$, suppose that $r=v_F(\lambda)$ is a multiple of $\ell^n=\per(\al)$. Write $\ell^m=\per(\al_E)$. Assume the following properties hold for the residue field $E_0$ of $E$:

\begin{enumerate}
    \item $E_0$ is Rost $\ell^m$-divisible.
    \item The corestriction map $H^3(E_0)\to H^3(k)$  is injective.
\end{enumerate}

Then $\al\cup (\lambda)=0$ implies $\lambda\in \prod^m_{i=0}\Nrd(\ell^{m-i}\al)^{\ell^{m-i}}\subseteq \cS(\al)=\prod^n_{i=0}\Nrd(\ell^{n-i}\al)^{\ell^{n-i}}$.
\end{lemma}
\begin{proof}
  Replacing $\lambda$ by $\lambda.\pi^{-r}$, we may assume $r=v_F(\lambda)$ is zero. Then the hypothesis $\al\cup(\lambda)=0$ means $\lambda\in \cR(\al)\cap U_F$. By \eqref{4p4p3v2},  $\lambda=N_{E/F}(\mu)$ for some unit $\mu$ in $E$. Using the projection formula for corestrictions, we obtain
  \[
    \Cor_{E/F}(\al_E\cup \mu)=\al\cup N_{E/F}(\mu)=\al\cup(\lambda)=0\;\in\;H^3(F)\,.
  \]Notice that $\al_E=\al'_E\in\Br(E)$ is unramified. So the cohomology class $\al_E\cup(\mu)\in H^3(E)$ is unramified and we have
   $\Cor_{E_0/k}(\ov{\al}_E\cup\bar\mu)=0\in H^3(k)$. By the corestriction injectivity assumption, $\ov{\al}_E\cup (\bar \mu)=0$ and hence $\al_E\cup \mu=0\in H^3(E)$. Since the residue field $E_0$ of $E$ is Rost $\ell^m$-divisible by assumption, we deduce from Lemma\;\ref{4p3nnn}  (2) that $\mu\in \cS(\al_E)=\prod^m_{i=0}\Nrd(\ell^{m-i}\al_E)^{\ell^{m-i}}$. It then follows immediately that
   $\lambda=N_{E/F}(\mu)\in \prod^m_{i=0}\Nrd(\ell^{m-i}\al)^{\ell^{m-i}}$.
\end{proof}

The proof of Theorem\;\ref{1p4nnn} is now immediate.

\begin{proof}[{\bf Proof of Theorem\;\ref{1p4nnn}}]\label{proof1p4}
    Let $\al\in \Br(F)$ be a Brauer class of period $\ell$. Assume the two conditions (Rost $\ell$-divisibility and $H^3$-corestriction injectivity) in the theorem hold. In view of Corollary\;\ref{4p3v2}, we need only to prove the following statement: If $\lambda\in U_F$ and $\al\cup(\lambda)=0\in H^3(F)$, then $\lambda\in \cS(\al)=(F^*)^{\ell}\cdot\Nrd(\al)$.

    By  \eqref{4p4p3v2}, we have $\lambda\in N_{E/F}(E^*)$. If $\al_E=0$, then $N_{E/F}(E^*)\subseteq\Nrd(\al)$ and we are done. Otherwise $\per(\al_E)=\ell$ and it suffices to apply Lemma\;\ref{4p4nnn}.
\end{proof}

When $\per(\al)=\ell$, we have in fact a description of the quotient group $\cR(\al)/\cS(\al)$ which holds without assuming the two conditions of Theorem\;\ref{1p4nnn}. To show this, we need the following variant of \cite[Prop.\;2.3]{Merkurjev95ProcSympos58}:

\begin{prop}\label{4p6v2}
  With notation as in $\eqref{4p1p4}$, suppose that $\per(\al)=\ell$. Then
  \[
  \cS(\al)\cap U_F=\big(N_{E/F}\Nrd(\al'_E)\cdot (U_F)^{\ell}\big)\cap U_F=\{a\in U_F\,|\, \ov{a}\in N_{E_0/k}\big(\cS(\ov{\al}'_{E_0})\big)\cdot k^{*\ell}\}\,.
  \]
\end{prop}

Following the same ideas as in \cite[\S2]{Merkurjev95ProcSympos58}, we will base our proof of Prop.\;\ref{4p6v2} (see page \pageref{proof4p6v2}) on Lemma\;\ref{4p7v2} and Prop.\;\ref{4p8v2} below.

\begin{lemma}\label{4p7v2}
  Let $F$ be a henselian discrete valuation field, $\pi\in F$ a uniformizer and $L/F$ a cyclic unramified extension of degree $m$. Let $\sigma$ be a generator of the Galois group $\mathrm{Gal}(L/F)$ and let $\Delta(L/F,\,\sigma,\,\pi)$ be the $F$-algebra generated by $L$ and an element $u$ subject to the following relations:
  \[
  u^m=\pi\quad\text{and}\quad a u=u\sigma(a)\; \text{ for all }\; a\in L\,.
  \]Let $\cA$ be an Azumaya algebra over the valuation ring $\mathcal{O}_F$. Write $\cB=\cA\otimes_{\mathcal{O}_F}\mathcal{O}_L$ and $D=\cA\otimes_{\mathcal{O}_F}\Delta(L/F,\,\sigma,\,\pi)$.

  \begin{enumerate}
    \item Let $\cB[u]$ be the $\mathcal{O}_F$-subalgebra  of $D$ generated by $\cB$ and $u$. Then $D=\cB[u]\otimes_{\mathcal{O}_F}F$.
    \item Let $d=\sum^{m-1}_{i=0}d_iu^i\in\cB[u]$ with each $d_i\in \cB$. Then
    \[
    \Nrd_D(d)\in \mathcal{O}_F\quad\text{and}\quad \Nrd_D(d)\equiv N_{L/F}\left(\Nrd_{\cB}(d_0)\right)\pmod{\pi\mathcal{O}_F}\;.
    \]
  \end{enumerate}
\end{lemma}
\begin{proof}
  (1) is clear from the constructions of $\cB$ and $D$. The proof of (2) is an easy adaption of the proof of \cite[Lemma\;2.1]{Merkurjev95ProcSympos58}. 
\end{proof}

\begin{prop}\label{4p8v2}
With notation and hypotheses as in Lemma\;$\ref{4p7v2}$, let $L_0/k$ be the residue field extension of $L/F$ and suppose that $m=[L:F]$ is not divisible by $\car(k)$.

Then
\[
\begin{split}
\left(\Nrd(D)\cdot F^{*m}\right)\cap U_F&=\left(N_{L/F}\Nrd(\cA\otimes L)\cdot (U_F)^m\right)\cap U_F\\
&=\{a\in U_F\,|\,\ov{a}\in N_{L_0/k}\Nrd(\cA\otimes L_0)\cdot k^{*m}\}\,.
\end{split}
\]
\end{prop}
\begin{proof}
  Let $G_1,\,G_2$ and $G_3$ be the first, second and third groups in the statement of the proposition.

  Note that $D\otimes L$ is Brauer equivalent to $\cA\otimes L$ by construction. Hence,
  \[
  N_{L/F}\Nrd(\cA\otimes L)=N_{L/F}\Nrd(D\otimes L)\subseteq \Nrd(D)\,.
  \]This gives the inclusion $G_2\subseteq G_1$.

  To see $G_3\subseteq G_2$, suppose $a\in G_3$ and write $\ov{a}=N_{L_0/k}(\beta) \gamma^m$ with $\gamma\in k^*$ and $\beta\in \Nrd(\cA\otimes L_0)$. There exists a splitting field $M_0/L_0$ of $\cA\otimes L_0$ such that $\beta=N_{M_0/L_0}(\beta')$ for some $\beta'\in M_0^*$. Let $M$ be the unramified extension of $L$ with residue field $M_0$. Then $M$ splits $\cA\otimes L$, so that $N_{M/L}(M^*)\subseteq \Nrd(\cA\otimes L)$. If $b'\in U_M$ is a lifting of $\beta'$ and $c\in U_F$ is a lifting of $\gamma$, then
  \[
  a^{-1}N_{L/F}N_{M/L}(b')c^m\;\subseteq\; U_F^1=\{x\in F^*\,|\,v_F(x-1)\ge 1\}\,.
  \]By Hensel's lemma we have $U^1_F\subseteq (F^*)^m$. So we get
  \[
  a\in N_{L/F}N_{M/L}(b')c^m \cdot (F^{*m})\subseteq N_{L/F}\Nrd(\cA\otimes L)\cdot (F^*)^{m}
  \]proving $a\in G_3$ as desired.

  Now suppose $b\in G_1=\left(\Nrd(D)\cdot F^{*m}\right)\cap U_F$. As in the proof of \cite[Prop.\;2.2]{Merkurjev95ProcSympos58}, we may assume $b=\Nrd(d)g^{-m}$ for some $d=\sum_{i\ge 0}d_iu^i\in \cB[u]$ (with $\cB=\cA\otimes\mathcal{O}_L$ as in Lemma\;\ref{4p7v2}) and $g\in U_F$. Using Lemma\;\ref{4p7v2} (as a substitute for \cite[Lemma\;2.1]{Merkurjev95ProcSympos58}), we find that
  \[
  \ov{b}=N_{L_0/k}\left(\Nrd(\ov{d}_0)\right)\ov{g}^{-m}\;\in\;N_{L_0/k}\Nrd(\cA\otimes L_0)\cdot k^{*m}\,,
  \]where $\ov{d}_0$ is the canonical image in $\cA\otimes L_0$ of $d_0\in \cB=\cA\otimes \mathcal{O}_L$. We have thus obtained the inclusion $G_1\subseteq G_3$.
\end{proof}

\begin{proof}[{\bf Proof of Proposition\;\ref{4p6v2}}]\label{proof4p6v2}
   Since we have assumed $\per(\al)=\ell$, in the notation of \eqref{4p1p4} we have either $E=F$ or $[E:F]=\ell$.
   Let $H_1,\,H_2$ and $H_3$ be the first, second and third groups in the statement of the proposition. Note that
   \[
   N_{E_0/k}\big(\cS(\ov{\al}'_{E_0})\big)\cdot k^{*\ell}=N_{E_0/k}\Nrd(\ov{\al}'_{E_0})\cdot k^{*\ell}
   \] since both $[E_0:k]$ and $\per(\ov{\al}'_{E_0})$ divide $\ell$. So we have
   \[
   H_3=\{a\in U_F\,|\, \ov{a}\in N_{E_0/k}\Nrd(\ov{\al}'_{E_0})\cdot k^{*\ell}\}\,.
   \]

As in the proof of Prop.\;\ref{4p8v2}, we have $H_1\supseteq H_2\supseteq H_3$. So it suffices to prove $H_1\subseteq H_3$.

If $[E:F]=\ell$, we take $\cA$ to be an Azumaya algebra over $\mathcal{O}_F$ that represents $\al'$ and let $L=E$. Then in Lemma\;\ref{4p7v2} and Prop.\;\ref{4p8v2} the algebra $D$ represents the Brauer class $\al$. Since the group $\cS(\al)$ is equal to $\Nrd(\al)F^{*\ell}$ in the current situation, the result follows from Prop.\;\ref{4p8v2}.

   If $E=F$, then $\al=\al'=\al'_E$ and
   \[
   H_3=\{a\in U_F\,|\,\ov{a}\in \Nrd(\ov{\al}')\cdot k^{*\ell}\}=\{a\in U_F\,|\,\ov{a}\in \cS(\ov{\al}')\}\,.
   \]Now the desired equality $H_1=H_3$ follows from the second equality in \eqref{4p4p2v2}.
\end{proof}

\begin{thm}\label{4p9v2}
  With notation as in $\eqref{4p1p4}$, suppose that $\per(\al)=\ell$.

  Then we have an isomorphism of groups
  \[
  \frac{\cR(\al)}{\cS(\al)}\cong \frac{\cR(\ov{\al}')\cap N_{E_0/k}(E^*_0)}{N_{E_0/k}\left(\cS(\ov{\al}'_{E_0})\right)\cdot k^{*\ell}}
  \]where $E_0$ denotes the residue field of $E$ and $\ov{\al}'\in\Br(k)$ is the canonical image of the unramified Brauer class $\al'$.
\end{thm}
\begin{proof}
It is easily seen that  the natural map
  \[
  \frac{\{a\in U_F\,|\,\ov{a}\in \cR(\ov{\al}')\cap N_{E_0/k}(E^*_0)\}}{\{a\in U_F\,|\,\ov{a}\in N_{E_0/k}\left(\cS(\ov{\al}'_{E_0})\right)\cdot k^{*\ell}\}}\lra \frac{\cR(\ov{\al}')\cap N_{E_0/k}(E^*_0)}{N_{E_0/k}\left(\cS(\ov{\al}'_{E_0})\right)\cdot k^{*\ell}}\;;\quad a\longmapsto \ov{a}
  \]is an isomorphism. Thus, combining Corollary\;\ref{4p3v2}, \eqref{4p4p3v2} and Proposition\;\ref{4p6v2} finishes the proof.
\end{proof}

If the corestriction map $\Cor_{E_0/k}: H^3(E_0)\to H^3(k)$ is injective, then  the group $\cR(\ov{\al}')\cap N_{E_0/k}(E^*_0)$ is equal to
$N_{E_0/k}\left(\cR(\ov{\al}'_{E_0})\right)$. If moreover $\cR(\ov{\al}_{E_0})=\cS(\ov{\al}'_{E_0})$, we see immediately from Theorem\;\ref{4p9v2} that $\cR(\al)=\cS(\al)$. Therefore, Theorem\;\ref{4p9v2} is a generalization of Theorem\;\ref{1p4nnn}.

\

The rest of this section is devoted to the proof of Theorem\;\ref{1p6nnn}.

\medskip

\newpara\label{4p5nnn} Throughout what follows, we keep the notation and hypotheses of Theorem\;\ref{1p6nnn}. Let $\al\in \Br(F)[\ell^n]$ and suppose $\lambda\in F^*$ lies in the Rost kernel of $\al$, i.e.,
\[
\al\cup(\lambda)=0\in H^3(F)=H^3(F,\,\Q_{\ell}/\Z_{\ell}(2))\,.
 \]
 We fix a uniformizer $\pi$ of $F$ and decompose $\al$ into $\al=\al'+(E/F,\,\pi)$, where $\al'\in \Br(F)$ is unramified and $E/F$ is a cyclic unramified extension whose degree divides $\ell^n$ . Note that the residue field extension $E_0/k$ of $E/F$ determines the residue $\partial(\al)\in H^1(k)$ of $\al$.
Also, we  write
\[
\lambda=\theta.(-\pi)^{s\ell^m}\,, \quad \text{with }\; s\in\Z\,,\,m\in\N\,,\,\theta\in U_F=\{a\in F^*\,|\,v_F(a)=0\}\,.
\]

\medskip

Our goal is to prove $\lambda\in\cS(\al)$.

\

We first explain a few reductions for the proof.
\begin{enumerate}
\item Without loss of generality, we may and we will assume $\per(\al)=\ell^n$.
  \item We may assume $v_F(\lambda)=\ell^m$ with $0\le m<n$.

Indeed, we have already seen that this is true if $\per(\al)=\ell^n\,|\,v_F(\lambda)$ (cf. Lemma\;\ref{4p4nnn}). Thus, we may assume $v_F(\lambda)=\ell^ms$ with $0\le m<n$ and $s\notin\ell\Z$. Choose $a,\,b\in \Z$ such that $sa+b\ell^{n-m}=1$. Let $\lambda_1=\lambda^a.\pi^{b\ell^n}$. Then
\[
v_F(\lambda_1)=\ell^msa+b\ell^n=\ell^m\quad\text{and}\quad \al\cup(\lambda_1)=a.\al\cup\lambda=0\,.
\]
If we know $\lambda_1\in \cS(\al)$, then we can deduce that $\lambda^a\in \cS(\al)$. Since the quotient group $F^*/\cS(\al)$ is $\ell$-primary torsion, this will imply $\lambda\in \cS(\al)$ as desired.
  \item In the remainder of this section, we shall prove $\lambda\in\cS(\al)$ by induction on $m$. The case $m=0$ has already been treated in Lemma\;\ref{4p2nnn}, which is valid even without assuming $\cd_{\ell}(k)\le 2$ nor $\mu_{\ell^n}\subseteq k$.

So we will assume $1\le m<n$ from now on.
  \item We may further assume $\ell^m\al'\neq 0$.

  Note that if $\per(\al')\le v_F(\lambda)=\ell^m$, we have
\[
\begin{split}
  0&=\al\cup(\lambda)=\al\cup (\theta)+\ell^m\al\cup (-\pi)\\
  &=\al\cup(\theta)+\ell^m\al'\cup (-\pi) \quad\text{ since } (E/F,\,\pi)\cup (-\pi)=0\\
  &=\al\cup(\theta)\,,\quad\text{since } \ell^m\al'=0 \text{ in this case}.
\end{split}
\]Thus, by Lemma\;\ref{4p4nnn} we have $\theta\in \cS(\al)$. On the other hand, since $\ell^m\al\cup(-\pi)=\ell^m\al'\cup (-\pi)=0$, by induction on the period, we obtain $-\pi\in \cS(\ell^m\al)$, whence $(-\pi)^{\ell^m}\in \cS(\ell^m\al)^{\ell^m}\subseteq\cS(\al)$. Hence $\lambda=(-\pi)^{\ell^m}\theta\in\cS(\al)$, proving what we want.
\item We may assume $\lambda\notin F^{*\ell}$.

Indeed, if $\lambda=\lambda_1^{\ell}$ for some $\lambda_1\in F^*$, then $\ell\al\cup \lambda_1=0$ and $v_F(\lambda_1)=\ell^{m-1}<\ell^m$. By induction on $m$ we have $\lambda_1\in\cS(\ell\al)$ and hence $\lambda\in \cS(\ell\al)^{\ell}\subseteq\cS(\al)$.
\end{enumerate}

\newpara\label{4p6nnn} With notation and hypotheses as above, by an {\bf \emph{inductive pair}} of $(\al,\,\lambda)$ we mean a degree $\ell$ extension $L/F$ together with an element $\mu\in L^*$ such that $\lambda=N_{L/F}(\mu)$ and $\mu\in\cS(\al_L)$. Clearly, if such a pair exists, we will have $\lambda\in \cS(\al)$.

To obtain an inductive pair, our basic strategy is the following: If $L/F$ is an unramified extension of degree $\ell$ with residue field $L_0/k$ and if $\xi\in U_L$ is an element such that
\begin{equation}\label{4p6p1}
\theta=N_{L/F}(\xi)\quad\text{and}\quad \ell^{m-1}\ov{\al}'_{L_0}=(E_0L_0/L_0\,,\,\ov{\xi})\;\in\;\Br(L_0)
\end{equation}then the pair $(L/F,\,\mu:=(-\pi)^{\ell^{m-1}}\xi)$ is an inductive pair of $(\al,\,\lambda)$.

To prove this, first observe that the element $\mu=(-\pi)^{\ell^{m-1}}\xi$ clearly satisfies $N_{L/F}(\mu)=\lambda$. Moreover, computing the residue of $\al_L\cup\mu$ and taking \eqref{4p6p1} into account we get
\[
\begin{split}
  \partial_L(\al_L\cup\mu)&=\partial_L\big(\al'_L\cup (\mu)\big)+\partial_L\big((EL/L,\,\pi)\cup (\mu)\big)\\
  &=\ell^{m-1}.\bar{\al}'_{L_0}+\partial_L\big((EL/L,\,\pi)\cup(\xi)\big)\\
  &=\ell^{m-1}\bar{\al}'_{L_0}-(E_0L_0/L_0,\,\bar{\xi})=0\,.
\end{split}
\]Since the residue field $L_0$ has cohomological $\ell$-dimension $\le 2$, it follows that $\al_L\cup(\mu)=0$. Noticing that $v_L(\mu)=\ell^{m-1}<\ell^m$, the induction hypothesis implies $\mu\in \cS(\al_L)$. This proves our claim.

\

\newpara\label{4p7nnn} Consider the canonical image $\bar\theta\in k$ of $\theta\in U_F$. Suppose that there exists a (separable) degree $\ell$ extension $L_0/k$ and an element $\xi_0\in L_0^*$ such that
\[
\bar\theta=N_{L_0/k}(\xi_0)\quad\text{and}\quad \ell^{m-1}\bar{\al}'_{L_0}=(E_0L_0/L_0,\,\xi_0)\;\in\;\Br(L_0)\,.
\]Then we may take $L/F$ to be the unramified extension with residue field $L_0/k$. Let $\xi_1\in U_L$ be a lifting of $\xi_0\in L_0^*$. We have
\[
\ov{N_{L/F}(\xi_1)}=N_{L_0/k}(\bar\xi_1)=N_{L_0/k}(\xi_0)=\bar\theta\,.
\]Hence, $\theta=N_{L/F}(\xi_1)\rho$ for some $\rho\in U^1_F$. By Hensel's lemma, $\rho=\rho_1^{\ell}$ for some $\rho_1\in U_F^1$. Putting $\xi=\xi_1\rho_1$ we obtain
\[
\begin{split}
  &\bar\xi=\bar\xi_1\cdot\bar\rho_1=\xi_0\cdot 1=\xi_0\\
\text{and}\quad & \theta=N_{L/F}(\xi_1)\rho_1^{\ell}=N_{L/F}(\xi_1\rho_1)=N_{L/F}(\xi)\,.
\end{split}
\]This means that $(L/F,\,\xi)$ is a pair satisfying \eqref{4p6p1}.

Thus, the proof of our main result is reduced to a problem that only involves data over the residue field $k$: the image $\bar\theta\in k$ of $\theta\in U_F$, the residue Brauer class $\bar{\al}'\in \Br(k)$ and the cyclic extension $E_0/k$ determined by the residue $\partial(\al)\in H^1(k)$ of $\al$.

Notice that the assumption $\al\cup\lambda=0$ at the beginning implies that $\ell^m.\bar\al'=(E_0/k,\,\bar\theta)$. Also, we have assumed $\lambda\notin F^{*\ell}$, so that $\bar\theta\notin k^{*\ell}$.

\

Changing notation, we are left to prove the following key lemma:

\begin{lemma}\label{4p8nnn}
  Let $k$ be a henselian discrete valuation field with residue field $k_0$. Assume that $\mathrm{char}(k_0)\neq\ell$ and $\cd_{\ell}(k)=2$. Let $\beta\in \Br(k)[\ell^n]$ be such that $0\neq\ell^m\beta=(K/k,\,\theta)$, where $1\le m<n$, $\theta\in k^*\setminus k^{*\ell}$ and $K/k$ is a cyclic extension with $[K:k]\,|\,\ell^n$.  Suppose that $\mu_{\ell^n}\subseteq k$.

  Then there exists a degree $\ell$ extension $L/k$ and an element $\xi\in L^*$ such that
  \begin{equation}\label{4p8p1}
\theta=N_{L/k}(\xi)\quad\text{and}\quad    \ell^{m-1}\beta_L=(KL/L,\,\xi)\;\in\;\Br(L)\,.
  \end{equation}
\end{lemma}
\begin{proof}
  If $v_k(\theta)\notin \ell\Z$, then $L:=k(\sqrt[\ell]{-\theta})$ is a degree $\ell$ extension of $k$ and the element $\xi:=-\sqrt[\ell]{-\theta}$ satisfies $N_{L/k}(\xi)=\theta$. The extension $L/k$ is totally ramified, so the residue field $L_0$ of $L$ is the same as the residue field $k_0$ of $k$. In the commutative diagram with exact rows (cf. \cite[p.21, Prop.\;8.6]{Serre03inGMScohinv})
  \[
  \begin{CD}
    0 @>>> H^2(L_0) @>>> H^2(L) @>{\partial}>> H^1(L_0) @>>> 0 \\
    && @V{e(L/k)\mathrm{Cor}}VV @VV{\mathrm{Cor}}V @VV{\mathrm{Cor}}V \\
    0 @>>> H^2(k_0) @>>> H^2(k) @>{\partial}>> H^1(k_0) @>>> 0 \\
  \end{CD}
  \]we have $H^2(L_0)=H^2(k_0)=0$ since $\cd_{\ell}(k_0)\le 1$. So, the corestriction map $\mathrm{Cor}_{L/k}: H^2(L)\to H^2(k)$ is an isomorphism. From this we obtain
  $\ell^{m-1}\beta_L=(KL/L,\,\xi)$, since
  \[
  \Cor_{L/k}(\ell^{m-1}\beta_L)=\ell^m\beta=(K/k,\,\theta)=\Cor_{L/k}((KL/L,\,\xi))\,.
  \]

  Next let us assume $\ell\,|\,v_k(\theta)$. We shall construct an unramified extension $L/k$ of degree $\ell$ together with an element $\xi\in L^*$ such that
  \begin{equation}\label{4p8p2}
    \theta=N_{L/k}(\xi)\quad\text{and}\quad \partial_L((KL/L,\,\xi)) \text{ equals } \partial(\ell^{m-1}\beta)_{L_0}(=\partial_L(\ell^{m-1}\beta_L))\;.
  \end{equation}
  This condition is equivalent to \eqref{4p8p1} by the assumption on the cohomological dimension.

  We choose a uniformizer $x$ of $k$ and write
  \[
  \theta=\theta_0x^{\ell s},\quad\text{ where } s\in\Z \text{ and } v_k(\theta_0)=0.
   \]By associating to each cyclic extension $M_0/k_0$ of $\ell$-power degree the Brauer class of the cyclic algebra $(M/k,\,x)$, $M/k$ denoting the unramified extension with residue field $M_0/k_0$, we get an inverse of the isomorphism $\partial: H^2(k)\to H^1(k_0)$. In particular, we may write $\beta=(M/k,\,x)$ for some unramified cyclic extension $M/k$ of degree dividing $\ell^n$. By Kummer theory (and using the assumption $\mu_{\ell^n}\subseteq k$), we have
   \[
   M=k(\sqrt[\ell^n]{b})\quad\text{ for some } b\in k^*\;\text{ with } v_k(b)=0\,.
    \](Note that we have $v_k(b)=0$ since $M/k$ is unramified.) Using a fixed primitive $\ell^n$-th root of unity, we may write $\beta$ in the form of a symbol algebra: $\beta=(b,\,x)_{\ell^n}$.
   Similarly, we may write
   \[
   K=k(\sqrt[\ell^n]{a})\quad\text{ for some }\; a\in k^*\;\text{ with } 1\le v_k(a)=\ell^t\le \ell^n\;.
   \]

   Suppose $L/k$ is an unramified extension of degree $\ell$ for which we can find an element $\bar\xi_0$ in the residue field $L_0$ such that $\bar{\theta}_0=N_{L_0/k_0}(\bar\xi_0)$. Then there exists a lifting $\xi_0\in U_L$ of $\bar\xi_0$ with $N_{L/k}(\xi_0)=\theta_0$. (The basic idea for the proof of this statement has been discussed in \eqref{4p7nnn}.) Thus, $\theta=\theta_0x^{\ell s}$ is the norm of $\xi:=\xi_0 x^s\in L$. Computation now shows
   \[
   \begin{split}
     \partial_L((KL/L,\,\xi))&=(-1)^{v(a)s}\bar\xi_0^{-v(a)}\bar{a}_0^{-s}\;\in\; L^*_0/L_0^{*\ell^n}\,,\\
     \text{and}\quad \partial(\ell^{m-1}\beta)_{L_0}&=\bar b^{\ell^{m-1}}\;\in\; L^*_0/L_0^{*\ell^n}\,,
   \end{split}
   \]where $a_0:=x^{v(a)}/a$. (Here we use a primitive root of unity to identity $H^1(L_0)$ with $L^*/\ell^n$.) Therefore, to have condition \eqref{4p8p2} satisfied, it suffices to find a degree $\ell$ extension $L_0/k_0$ and $\bar\xi_0\in L_0^*$ with $N_{L_0/k_0}(\bar\xi_0)=\bar\theta_0$ such that
   \begin{equation}\label{4p8p3}
     (-1)^{v(a)s}\bar b^{\ell^{m-1}}\cdot \bar\xi_0^{v(a)}\bar{a}_0^{s}\;\in\; L_0^{*\ell^n}\,.
   \end{equation}

   Notice that the assumption $\ell^m\beta=(K/k,\,\theta)=(a,\,\theta)_{\ell^n}$ yields
   \[
   \bar b^{\ell^m}=\partial(\ell^m\beta)=(-1)^{v(a)s\ell}\bar\theta_0^{-v(a)}\bar a_0^{-s\ell}\;\in\; k^*_0/k_0^{*\ell^n}
   \]i.e.,
   \begin{equation}\label{4p8p4}
     (-1)^{v(a)s\ell}\bar b^{\ell^m}\cdot \bar\theta_0^{v(a)}\bar a_0^{s\ell}\;\in\; k_0^{*\ell^n}\;.
   \end{equation}
   If $\ell^t=v(a)=1$ (i.e. $t=0$), this contradicts the assumption $\theta_0x^{\ell s}=\theta\notin k^{*\ell}$. So we have $t\ge 1$ and $\ell\,|\,v(a)=\ell^t$. In particular, $(-1)^{v(a)s}\in k_0^{*\ell^n}$. Hence, we may ignore the $(-1)$-powers in \eqref{4p8p3} and \eqref{4p8p4}.

   From \eqref{4p8p4} we get
   \[
    \text{the order of } \bar b^{\ell^m}\bar a_0^{\ell s}  \text{ in }\; k_0^*/\ell^n =\text{the order of } \bar\theta_0^{\ell^t} \text{ in } k^*_0/k_0^{*\ell^n} =\ell^{n-t}
   \]since $\theta_0\notin k_0^{*\ell}$ (and $k_0$ contains enough roots of unity). It follows that $\bar b^{\ell^m}\cdot \bar a_0^{\ell s}\in k_0^{*\ell^t}\setminus k_0^{*\ell^{t+1}}$, and hence
   \[
   \bar b^{\ell^{m-1}}\bar a_0^{s}=\bar{b}_0^{\ell^{t-1}}\,,\quad\text{for some } \bar b_0\in k_0^*\setminus k_0^{*\ell}\;.
   \]
   Now set
   \[
   L_0:=k_0(\sqrt[\ell]{\bar b_0})\,,\;\bar c_0=\sqrt[\ell]{\bar b_0}\in L_0^*\quad\text{and}\quad \bar c=\bar c_0^{\ell^{t-1}}\,.
   \]Then $\bar c^{\ell}=\bar b_0^{\ell^{t-1}}=\bar b^{\ell^{m-1}}\bar a_0^s$. Since $\cd_{\ell}(k_0)\le 1$, the norm map $N: L_0^*\to k_0^*$ is surjective. So we can find $\bar\xi_1\in L_0^*$ such that $N(\bar\xi_1)=\bar\theta_0$. Then
   \[
   N(\bar c\bar\xi_1^{\ell^{t-1}})=N(\bar c_0)^{\ell^{t-1}}\bar\theta_0^{\ell^{t-1}}=\pm \bar b_0^{\ell^{t-1}}\bar\theta^{\ell^{t-1}}_0=\pm \bar b^{\ell^{m-1}}\bar a_0^s\cdot\bar\theta_0^{\ell^{t-1}}\;\in \;k_0^{*\ell^{n-1}}
   \]by \eqref{4p8p4}. (Here if $\ell=2$, then $-1\in k_0^{*\ell^{n-1}}$ since we assumed $\mu_{\ell^n}\subseteq k_0$.) Writing $T=\ker(N: L_0^*\to k_0^*)$, we can deduce from the commutative diagram
   \[
  \begin{CD}
    0 @>>> T @>>> L_0^* @>{N}>> k_0^* @>>> 0 \\
    && @V{\ell^{n-1}}VV @VV{\ell^{n-1}}V @VV{\ell^{n-1}}V \\
    0 @>>> T @>>> L_0^* @>{N}>> k_0^* @>>> 0 \\
 && @VVV @VVV @VVV \\
&& T/\ell^{n-1} @>>> L_0^*/\ell^{n-1} @>{N}>> k_0^*/\ell^{n-1} @>>> 0 \\
  \end{CD}
  \]that
  \begin{equation}\label{4p8p5}
    \bar c.\bar\xi_1^{\ell^{t-1}}.\bar\rho\;\in \;L_0^{*\ell^{n-1}}\quad\text{ for some } \bar\rho\;\in\;T\,.
  \end{equation}
  Note that $\bar c=\bar c_0^{\ell^{t-1}}$. Thus, \eqref{4p8p5} implies $\bar\rho=\bar\rho_0^{\ell^{t-1}}$ for some $\bar\rho_0\in L_0^*$. Since $N(\bar\rho_0)^{\ell^{t-1}}=N(\rho)=1$, we have $N(\bar\rho_0)=\eta^{-\ell}=N(\eta^{-1})$ for some $\eta\in \mu_{\ell^t}\subseteq k_0^*$. Putting $\bar\rho_1=\bar\rho_0.\eta$ and $\bar\xi_0=\bar\xi_1.\bar\rho_1$, we get
  \[
  N(\bar\rho_1)=1\quad\text{and}\quad N(\bar\xi_0)=N(\bar\xi_1)=\bar\theta_0\,.
  \]
  Finally,
  \[
  \begin{split}
    \bar b^{\ell^{m-1}}.\bar a_0^s.\bar\xi_0^{\ell^t}&=\bar c^{\ell}.\bar\xi_0^{\ell^t}=\left(\bar c.\bar\xi_0^{\ell^{t-1}}\right)^{\ell}=\left(\bar c.\bar\xi_1^{\ell^{t-1}}.\bar\rho_1^{\ell^{t-1}}\right)^{\ell}\\
     &=\left(\bar c.\bar\xi_1^{\ell^{t-1}}.\bar\rho_0^{\ell^{t-1}}.\eta^{\ell^{t-1}}\right)^{\ell}=\left(\bar c.\bar\xi_1^{\ell^{t-1}}.\bar\rho\right)^{\ell}\cdot \eta^{\ell^t}\\
     &=\left(\bar c.\bar\xi_1^{\ell^{t-1}}.\bar\rho\right)^{\ell}\;,\quad \text{since }\eta\in \mu_{\ell^t}\;.
  \end{split}
  \]The last term belongs to $L_0^{*\ell^n}$, by \eqref{4p8p5}.   We have thus obtained a pair $(L_0/k_0,\,\bar\xi_0)$ satisfying \eqref{4p8p3}. So the proof is finished.
\end{proof}

\section{Tame classes of  period equal to residue characteristic}\label{sec5}

Our aim in this section is to state and prove a variant of Theorem\;\ref{1p4nnn} for Brauer classes of prime period $p$, when $p$ is the characteristic of the residue field $k$.

\medskip

\newpara\label{5p1nnn} Let us first recall a few facts about the Kato--Milne cohomology.

 Let $k$ be a field of characteristic $p>0$. Let $r\in\N$. The Kato--Milne cohomology group $H^{r+1}(k,\,\mathbb{Q}_p/\Z_p(r))$ was originally defined by using technical theories such as Bloch's groups \cite{Bloch77PubIHES} and de Rham--Witt theory (cf. \cite{Illusie79}, \cite{KatoII80},  \cite{Milne76AnnSciENS}). A simpler definition via Galois cohomology of Milnor $K$-groups of the separable closure is explained in \cite[Appendix\;A]{Merkurjev03inGMScohinv}.

We note that $H^1(k,\,\mathbb{Q}_p/\Z_p)$ is nothing but the $p$-primary torsion part of the usual Galois cohomology group $H^1(k,\,\mathbb{Q}/\Z)$ (with the Galois action on $\mathbb{Q}/\Z$ being trivial), and that $H^2(k,\,\mathbb{Q}_p/\Z_p(1))$ may be identified with the $p$-primary torsion part $\Br(k)[p^{\infty}]$ of the Brauer group $\Br(k)$.

Most useful in this paper is the  $p$-torsion part of $H^{r+1}(k,\,\mathbb{Q}_p/\Z_p(r))$, which we denote by $H^{r+1}(k,\,\Z/p\Z(r))$ or $H^{r+1}_p(k)$ and which we can describe easily using differential forms.  Let $\Omega^r_k$ denote the $r$-th exterior power (over $k$) of the space of absolute differential forms $\Omega_k:=\Omega_{k/\Z}$ (with the convention $\Omega^0_k=k$). Let $B^r_k$ be the image of the exterior differential map
$\mathrm{d}: \Omega^{r-1}_k\lra \Omega^r_k$ if $r\ge 1$ and put $B^0_k=0$. We have  the Artin--Schreier map
\[
\wp\;:\;\;\Omega^{r}_k\lra \Omega^{r}_k/B^r_k\;;\quad b\frac{\mathrm{d} a_1}{a_1}\wedge\cdots \wedge\frac{\mathrm{d}a_r}{a_r}\longmapsto (b^p-b)\frac{\mathrm{d} a_1}{a_1}\wedge\cdots \wedge\frac{\mathrm{d}a_r}{a_r}\;
\]and we may define
\[
H^{r+1}_p(k)=H^{r+1}(k,\,\Z/p\Z(r)):=\mathrm{coker}\left(\wp\;:\;\;\Omega^{r}_k\lra \Omega^{r}_k/B^r_k\right)\,.
\]
Under the natural identification $H^2_p(k)=\Br(k)[p]$, the differential form $b\frac{\mathrm{d}a}{a}$ corresponds to the Brauer class of the \emph{symbol $p$-algebra} $[b,\,a)_p$, the $k$-algebra generated by two elements $x,\,y$ subject to the relations
\[
x^p-x=b\,,\;y^p=a\quad\text{and }\; yx=(x+1)y\,.
\]
The cup product
\[
\cup\;:\;\;H^2_p(k)\times \left(k^*/k^{*p}\right)\lra H^3_p(k)\subseteq H^3(k,\,\mathbb{Q}_p/\Z_p(2))
\]satisfies
\[
\left(b\frac{\mathrm{d}a}{a}\right)\cup (\lambda)=b\frac{\mathrm{d}a}{a}\wedge\frac{\mathrm{d}\lambda}{\lambda}
\]for all $b\frac{\mathrm{d}a}{a}\in H^2_p(k)$ and $\lambda\in k^*$.

\

\newpara\label{5p2nnn}
Now let $F$ be a henselian excellent discrete valuation field with residue field $k$ of characteristic $\car(k)=p>0$. (Here $F$ may have characteristic $0$ or $p$, {and by saying $F$ is excellent we mean that its valuation ring is excellent}.) We define
\[
  H^{r+1}_p(F)_{tr}:=\ker\big(H^{r+1}(F,\,\Z/p\Z(r))\lra H^{r+1}(F_{nr}\,,\,\Z/p\Z(r))\big)\,,
\]where $F_{nr}$ denotes the maximal unramified extension of $F$. There is a natural inflation map (cf. \cite[p.659, $\S$3.2, Definition\;2]{KatoII80})
\[
  \mathrm{Inf}\;:\;\; H^{r+1}_p(k)\lra H^{r+1}_p(F)
\]such that
\[
  \mathrm{Inf}\left(b\frac{\mathrm{d}a_1}{a_1}\wedge\cdots\wedge\frac{\mathrm{d}a_r}{a_r}\right)=\tilde{b}\cup (\tilde{a}_1)\cup\cdots\cup (\tilde{a}_r)
\]where $\tilde{a}_i$ is any lifting of $a_i$, and $\tilde{b}\in H^1(F,\,\Z/p\Z)$ denotes the canonical lifting (i.e. inflation) of the character $\chi_b\in H^1(k,\,\Z/p)$ corresponding to the Artin--Schreier extension $k[T]/(T^p-T-b)$ of $k$. The choice of a uniformizer $\pi\in F$ defines a homomorphism (cf. \cite[p.659, $\S$3.2, Lemma\;3]{KatoII80})
\[
h_{\pi}\,:\;\; H^r_p(k)\lra H^{r+1}_p(F)\;;\quad w\longmapsto \mathrm{Inf}(w)\cup (\pi)\,.
\]The images of $\mathrm{Inf}$ and $h_{\pi}$ are both contained in $H^{r+1}_p(F)_{tr}$. It is proved in \cite[p.219, Thm.\;3]{Kato82} that the above two maps induces an isomorphism
\[
  \mathrm{Inf}\oplus h_{\pi}\;:\;\; H^{r+1}_p(k)\oplus H^{r}_p(k)\simto H^{r+1}_p(F)_{tr}\;.
\]
We can thus define a \emph{residue map} $\partial:H^{r+1}_p(F)_{tr}\to H^r_p(k)$, which fits into the following split exact sequence
\begin{equation}\label{5p2p1}
  0\lra H^{r+1}_p(k)\overset{\mathrm{Inf}}{\lra}H^{r+1}_p(F)_{tr}\overset{\partial}{\lra} H^r_p(k)\lra 0\,.
\end{equation}Moreover, we have analogs of the formulae \eqref{4p1p1} and \eqref{4p1p2} for the residue maps.

Let us specialize to the case $r=1$. Then $H^{r+1}_p(k)=H^{2}_p(k)$ is the $p$-torsion subgroup $\Br(k)[p]$ of $\Br(k)$ and similarly for $H^2_p(F)$. So in this case \eqref{5p2p1} reads
\begin{equation}\label{5p2p2}
  0\lra \Br(k)[p]\overset{\mathrm{Inf}}{\lra} \Br_{tr}(F)[p]\overset{\partial}{\lra} H^1(k,\,\Z/p)\lra 0\,,
\end{equation}where $\Br_{tr}(F)[p]:=\ker(\Br(F)[p]\lra \Br(F_{nr})[p])$  is identified with
\[
H^2_p(F)_{tr}=\ker(H^2_p(F)\lra H^2_p(F_{nr}))\,.
\]

We remark that \eqref{5p2p2} can be extended to  $p$-primary torsion subgroups and pieced together with \eqref{4p1p3} to give the following exact sequence
\[
  0\lra \Br(k)\lra \Br_{tr}(F)\lra H^1(k,\,\mathbb{Q}/\Z)\lra 0\,,
\]where
\[
\Br_{tr}(F):=\ker(\Br(F)\lra \Br(F_{nr}))\,
\]is referred to as the \emph{tame} or \emph{tamely ramified}  part of $\Br(F)$. (Perhaps a more appropriate name for this subgroup is the \emph{inertially split} part of $\Br(F)$, according to the terminology of \cite[$\S$5]{JW90} and \cite[$\S$3]{Wadsworth02}. In fact, the tame part and the inertially split part  can be defined for any henselian valued field. Fortunately, these two notions make no difference in our situation, i.e., for discrete valuation fields.)

\

Now we state our main result in this section. {For the construction and main properties of the corestriction maps needed below, we refer the reader to \cite[p.658]{KatoII80}.}

\begin{thm}\label{5p3nnn}
  Let $F$ be a henselian excellent discrete valuation field with residue field $k$ of characteristic $\car(k)=p>0$.  Suppose that the following properties hold for every finite cyclic extension $L/k$  of degree $1$ or $p$:
  \begin{enumerate}
  \item (\emph{Period equals index}) Every Brauer class of period $p$ over $L$ has index $p$.
    \item ($H^3$-\emph{corestriction injectivity})  The corestriction map
  \[
  \Cor_{L/k}\,:\;\;H^3_p(L)\lra H^3_p(k)
  \]is injective.
  \end{enumerate}

  Then for every $\al\in\Br_{tr}(F)[p]$ one has
  \[
  \{\lambda\in F^*\,|\,\al\cup (\lambda)=0\;\in\;H^3_p(F)\}=F^*{}^p\cdot\Nrd(\al)\,.
  \]
\end{thm}
\begin{proof}
  The basic strategy is identical to that of the proof of Theorem\;\ref{1p4nnn}. A slight difference is that we now assume the period-index condition to strengthen the Rost-divisibility hypothesis. This is needed to overcome the difficulty that the method of lifting $p$-th powers using Hensel's lemma fails.

  The tameness assumption on $\al$ allows us to consider its residue $\partial(\al)\in H^1(k,\,\Z/p)$, which we represent by a finite cyclic extension $E_0/k$ of degree 1 or $p$. We write $E/F$ for the unramified extension with residue field extension $E_0/k$. We fix a uniformizer $\pi\in F$ and use the split exact sequence \eqref{5p2p2} to write
  \[
  \al=\al'+(E/F\,,\,\pi)\quad\text{with }\al'\in \Br(F)[p]\;\text{ unramified},
  \]as an analogue of \eqref{4p1p4}. (Here we call a Brauer class $\beta\in\Br(F)[p]$ \emph{unramified} if $\beta\in \Br_{tr}(F)[p]$ and if $\partial(\beta)=0$. Equivalently, an unramified element of $\Br(F)[p]$ is an element in the image of the inflation map $\Br(k)[p]\to \Br_{tr}(F)[p]$.)

  Let $\lambda\in F^*$ be such that $\al\cup (\lambda)=0\in H^3_p(F)$. We want to show $\lambda\in F^{*p}\cdot\Nrd(\al)$. Replacing $\lambda$ with $\lambda^s\pi^{pc}$ for suitable integers $s,\,c\in\Z$ if necessary, we may assume that $v_F(\lambda)=0$ or $1$.

  If $v_F(\lambda)=1$, computing the residue of $\al\cup(\lambda)$ we see that $\al=(E/F,\,-\lambda)$ (cf. \eqref{4p1p6}). Thus $\al$ is split over $L:=F(\sqrt[p]{-\lambda})$ and
  \[
  (-1)^p\lambda=N_{L/F}(\sqrt[p]{-\lambda})\in N_{L/F}(L^*)\subseteq \Nrd(\al)\,.
  \]This shows $\lambda\in F^{*p}\cdot \Nrd(\al)$ as desired.

  Now assume $v_F(\lambda)=0$. As in the proof of Lemma\;\ref{4p4nnn}, we have $\lambda=N_{E/F}(\mu)$  for some unit $\mu$ in $E$ such that $\al_E\cup(\mu)=0\in H^3_p(E)$. Since $\al_E=\al'_E$ is unramified, the period-index assumption implies that $\al_E$ has index $p$ or $1$. By \cite[p.94, Thm.\;6]{Gille00}, $\mu$ is a reduced norm for $\al_E$ and it follows that $\lambda\in N_{E/F}(\mu)\in\Nrd(\al)$. This completes the proof.
\end{proof}

\begin{remark}\label{5p4nnn}
For a field $k$ of characteristic $p$, there are two useful variants of the cohomological $p$-dimension $\cd_p(k)$: the \emph{separable $p$-dimension} $\sd_p(k)$ (\cite[p.62]{Gille00}) and \emph{Kato's $p$-dimension} $\dim_p(k)$ (\cite[p.220]{Kato82}). They are defined as follows:
\[
\begin{split}
\sd_p(k)&:=\inf\{r\in\N\,|\,H^{r+1}_p(k')=0\,,\;  \text{for all finite separable extension } k'/k\}\,,\\
\dim_p(k)&:=\inf\{r\in\N\,|\,[k:k^p]\le p^r\text{ and } H^{r+1}_p(k')=0\,,\; \text{for all finite extension } k'/k\}\,.
\end{split}
\]
(In characteristic different from $p$, both the separable $p$-dimension and Kato's $p$-dimension are defined to be the same as the cohomological $p$-dimension.)

It is easy to see that $\sd_p(k)\le \dim_p(k)$. Moreover, it can be shown that
\[
\log_p[k:k^p]\le \dim_p(k)\le \log_p[k:k^p]+1\,.
\] Therefore, we have the following implications:
\[
\dim_p(k)\le 1 \Longrightarrow [k:k^p]\le p \Longrightarrow \dim_p(k)\le 2\Longrightarrow \sd_p(k)\le 2\,.
\]

Notice that Condition 2 in Theorem\;\ref{5p3nnn} is true when $\sd_p(k)\le 2$, and Condition 1 holds if $[k:k^p]\le p$ by a theorem of Albert (cf. \cite[Lemma\;9.1.7]{GilleSzamuely17}).

If in Theorem\;\ref{5p3nnn} we assume $\dim_p(k)\le 1$, then Condition 1 also holds, but the theorem  is not new. In fact, even more is true in that case. This is because we have $\dim_p(F)\le 2$ by \cite[Thm.\;3.1]{KatoKuzumaki}. In particular,  $\sd_p(F)\le 2$. So, by \eqref{1p1nnn} (2), $F$ is Rost $p^{\infty}$-divisible.
\end{remark}

We shall now show that the Rost $p$-divisibility of $F$ is still true when we only assume $[k:k^p]\le p$. This is based on the following result of Kato.

\begin{prop}[{\cite[p.337, $\S$3, Lemma\;5]{KatoI79}}]\label{5p5nnn}
  Let $F$ be a henselian excellent discrete valuation field with residue field $k$ of characteristic $p$. Assume that $[k:k^p]=p$. Then for every $\al\in \Br(F)[p]$ that is not in $\Br_{tr}(F)[p]$, one has $\ind(\al)=p$.
\end{prop}

\begin{coro}\label{5p6nnn}
   Let $F$ be a henselian excellent discrete valuation field with residue field $k$ of characteristic $p$.

   If $[k:k^p]\le p$ (e.g. $k=k_0(x)$ or $k_0(\!(x)\!)$ for some perfect field $k_0$), then $F$ is Rost $p$-divisible.
\end{coro}
\begin{proof}
  Combine \eqref{5p3nnn}, \eqref{5p4nnn} and \eqref{5p5nnn}.
\end{proof}

\begin{example}\label{5p7}
We give some further examples to which Theorem\;\ref{5p3nnn} applies, i.e., examples where the residue field $k$ satisfies the two conditions of the theorem.

We first note that by \cite[p.234, Prop.\;2 (2)]{KatoKuzumaki} (or \cite{ArasonBaeza2010}), if $k$ is a $C_2$-field of characteristic $p$, then $\dim_p(k)\le 2$, so that Condition 2 of \ref{5p3nnn} holds by Remark\;\ref{5p4nnn}. If $p=2$, Condition 1 is true for any $C_2$-field $k$ of characteristic $2$. (The period-index condition for 2-torsion Brauer classes can be shown easily by considering the Albert forms of biquaternion algebras.)

For general $p$, the field $k$ can be any of the following fields:

\begin{enumerate}
  \item $k$ is a field of transcendence degree 2 over an algebraically closed field (of characteristic $p$).

  In this case the period-index condition follows from \cite[Thm.\;4.2.2.3]{Lieb08}. The field $k$ is $C_2$ by the Tsen--Lang theorem.

  \item $k=k_0(\!(x)\!)$, where $k_0$ is a $C_1$-field of characteristic $p$.

  Here $k$ is $C_2$ by Greenberg's theorem (\cite{Greenberg66}). The ``period equals index'' property holds for wildly ramified classes in $\Br(k)[p]$, by Prop.\;\ref{5p5nnn}. For tame classes this follows from the split exact sequence \eqref{5p2p2} and the fact that $\Br(k_0)=0$.
  \item $k=\mathbb{F}(\!(x)\!)(\!(y)\!)$, where $\mathbb{F}$ is a finite field of characteristic $p$.

  The period-index condition follows from \cite[Corollary\;3.5]{AravireJacob95}. The injectivity of the corestriction map is proved in \cite[p.660, $\S$3.2, Prop.\;1]{KatoII80}.
\end{enumerate}

\end{example}

\section{A modified version of Suslin's conjecture}\label{sec6}

Our main results (Theorems\;\ref{1p4nnn} and \ref{1p6nnn}) provide new examples  of  fields of cohomological dimension 3 that satisfy Suslin's conjecture\;\ref{1p2nnn}. On the other hand,  there  exist counterexamples to Suslin's conjecture as we have said in the introduction. We shall now have a closer look at the known counterexamples.

\medskip

\newpara\label{6p1} Let $\ell$ be a prime. Let $k$ be a field of characteristic different from $\ell$ and suppose that $k$ contains a primitive $\ell$-th root of unity. Let $F=k(t_1,\dotsc, t_n)$ be a purely transcendental extension in $n$ variables over $k$. Let $a_1,\,\dotsc, a_n\in k^*$, $M_i=k(\sqrt[\ell]{a_i})$ and $L=M_1\cdots M_n=k(\sqrt[\ell]{a_1}\,,\dotsc, \sqrt[\ell]{a_n})$. Let $\al\in \Br(F)$ be the Brauer class of the tensor product $\otimes^n_{i=1}(a_i,\,t_i)$. By \cite[Prop.\;2.5]{Merkurjev95ProcSympos58}, a necessary condition for {$\cR(\al)=\cS(\al)$} is
\begin{equation}\label{6p1p1}
\bigcap^n_{i=1}N_{M_i/k}(M_i^*)=k^{*\ell}\cdot N_{L/k}(L^*)\,.
\end{equation}
For any odd prime $\ell$ and $n=2$, Merkurjev constructed in \cite[$\S$2]{Merkurjev95ProcSympos58} examples where \eqref{6p1p1} fails, whence an example with {$\cR(\al)\neq \cS(\al)$}.

If $\ell=2$ and $n=3$, \eqref{6p1p1} is encoded in the property $P_1(3)$ of fields defined and studied in \cite{Tignol81LNM} and \cite{ShapiroTignolWadsworth82JA} (see also \cite[p.741, Prop.\;3]{Gille97JA}). Thus, for a field $k$ that does not possess this property (cf. \cite[$\S$5]{ShapiroTignolWadsworth82JA}) one can find counterexamples to \eqref{6p1p1}.

Note however a simple observation: \eqref{6p1p1} holds trivially if the field $k$ satisfies $\cd_{\ell}(k)\le 1$. Therefore, in the above counterexamples we must have $\cd_{\ell}(k)>1$ and hence $\cd_{\ell}(F)>3$.

\

The above discussions together with the main results of this paper lead us to propose the following modification of Suslin's conjecture:

\begin{conj}\label{6p2}
  Let $\ell$ be a prime and $F$ a field with separable $\ell$-dimension $\sd_{\ell}(F)\le 3$ $($cf. Remark$\;\ref{5p4nnn})$. Then $F$ is Rost $\ell^{\infty}$-divisible.
\end{conj}

As a related question, one may ask whether every $C_3$-field is Rost divisible.

When $\ell\neq\car(F)$, Conjecture\;\ref{6p2} is true if for every Severi--Brauer variety $X$ associated to an $\ell$-power degree division algebra over $F$, the Chow groups $\mathrm{CH}^d(X),\,d\ge 1$ are all torsion free (cf. \cite[Prop.\;1.11]{Merkurjev95ProcSympos58}. In fact, it suffices to assume $\mathrm{CH}^d(X)$ is torsion free for $d\ge 3$.)

\medskip

{As we have mentioned in the introduction, in cohomological dimension 3 it may happen that the Rost kernel is strictly larger than the group of reduced norms. For example, it is shown in the proof of \cite[Theorem\;4]{Merkurjev91} that there exists a field $k$ of characteristic 0 and of cohomological dimension 2 such that some biquaternion algebra $D_0$ over $k$ is a division algebra. If $F=k(\!(t)\!)$ and $\al\in\Br(F)$ is the Brauer class of $D_0\otimes_kF$, then $F$ has cohomological dimension 3, the element $t^2\in F^*$ lies in the Rost kernel of $\al$, but it is not a reduced norm for $\al$ (\cite[Remark\;5.1]{CTPaSu}).}

\begin{remark}\label{6p3}
The rational function field $k=\C(x,\,y)$ does not have the property $P_1(3)$ (\cite[Cor.\;5.6]{ShapiroTignolWadsworth82JA}), so one can find quadratic extensions $M_i=k(\sqrt{a_i}),\,i=1,\,2,\,3$ such that \eqref{6p1p1} fails. Similar to \cite[Prop.\;2.5]{Merkurjev95ProcSympos58}, for the field $K=k(\!(t_1)\!)(\!(t_2)\!)(\!(t_3)\!)$ and the Brauer class $\al$ of $\otimes^3_{i=1}(a_i,\,t_i)$, one has {$\cR(\al)\neq \cS(\al)$} (cf. \cite[p.283]{KnusLamShapiroTignol95}).

So, the field $K=k(\!(t_1)\!)(\!(t_2)\!)(\!(t_3)\!)$, with $k=\C(x,\,y)$, is not Rost 2-divisible. This shows that Theorem\;\ref{1p4nnn} can be false if we drop the $H^3$-corestriction injectivity assumption, for otherwise it would imply by  induction that the field $K$ here is Rost 2-divisible.
\end{remark}

\

\noindent \emph{Acknowledgements.} This work is motivated by a question which Yong Hu heard from Mathieu Florence at a conference held in Florence, Italy, in December 2017. {We are indebted to an anonymous referee for valuable comments. In particular, establishing a result in the form of Theorem\;\ref{4p9v2} is kindly suggested by her/him.} The first author  is also grateful to Philippe Gille for helpful discussions. Yong Hu is supported by a grant from the National Natural Science Foundation of China (Project No. 11801260). Zhengyao Wu is supported by  Shantou University Scientific Research Foundation for Talents (Grant No. 130-760188) and the National Natural Science Foundation of China (Project No. 11701352).

\addcontentsline{toc}{section}{\textbf{References}}

\bibliographystyle{alpha}

\bibliography{Rost}

\

Yong HU

\medskip

Department of Mathematics

Southern University of Science and Technology

No. 1088, Xueyuan Blvd., Nanshan district

518055 Shenzhen, Guangdong,

P.R. China

Email: huy@sustech.edu.cn

\

Zhengyao WU

\medskip

Department of Mathematics

Shantou University

243 Daxue Road

515063 Shantou, Guangdong

P.R. China

Email: wuzhengyao@stu.edu.cn

\clearpage \thispagestyle{empty}

\end{document}